\documentclass{article}
\usepackage[utf8]{inputenc} %
\usepackage[T1]{fontenc}    %
\usepackage[hidelinks]{hyperref}       %
\usepackage{url}            %
\usepackage{booktabs}       %
\usepackage{fullpage}
\usepackage{amsmath}
\usepackage{amsfonts}       %
\usepackage{nicefrac}       %
\usepackage{microtype}      %
\usepackage[boxruled]{algorithm2e}
\newcommand{\RR}{\mathbf{R}}
\newcommand{\NN}{\mathbf{N}}
\newcommand{\cl}{\mathrm{cl}\;}
\newcommand{\range}{\mathrm{range}}
\newcommand{\prox}{\mathbf{prox}}
\newcommand{\supp}{\mathrm{supp}}
\newcommand{\proj}{\mathbf{proj}}

\newcommand{\pssize}{\eta}
\newcommand{\dssize}{\tau}
\DeclareMathOperator*{\argmin}{\arg\!\min}
\usepackage[utf8]{inputenc}
\usepackage[english]{babel}
\usepackage{graphicx}
\usepackage{mathtools}
\usepackage{color}
\usepackage{graphicx}
\usepackage{subcaption}
\usepackage{makecell}
\usepackage{multirow}
\usepackage{hhline}
\usepackage{diagbox}
\DeclarePairedDelimiterX{\dotp}[2]{\langle}{\rangle}{#1, #2}

\usepackage{amsthm}

\newtheorem{theorem}{Theorem}
\newtheorem{example}{Example}
\newtheorem{remark}{Remark}
\newtheorem{lemma}{Lemma}

\newtheorem{fact}{Fact}
\newtheorem{definition}{Definition}
\newtheorem{proposition}{Proposition}
\newtheorem{corollary}{Corollary}
\newcommand{\tz}{\tilde{z}}
\newcommand{\tx}{\tilde{x}}
\newcommand{\ty}{\tilde{y}}

\title{Infeasibility detection with primal-dual hybrid gradient for large-scale linear programming}
\author{David Applegate\thanks{Google Research ({\tt dapplegate@google.com})} \qquad Mateo D\'iaz\thanks{Cornell University Center for Applied Mathematics ({\tt md825@cornell.edu}). Work done while interning at Google Research.} \qquad Haihao Lu\thanks{University of Chicago Booth School of Business and Google Research ({\tt haihao.lu@chicagobooth.edu})} \qquad Miles Lubin\thanks{Google Research ({\tt mlubin@google.com})}}
\date{}

\begin{document}

\maketitle
\begin{abstract}
    We study the problem of detecting infeasibility of large-scale linear programming problems using the primal-dual hybrid gradient method (PDHG) of Chambolle and Pock (2011). The literature on PDHG has mostly focused on settings where the problem at hand is assumed to be feasible. When the problem is not feasible, the iterates of the algorithm do not converge. In this scenario, we show that the iterates diverge at a controlled rate towards a well-defined ray. The direction of this ray is known as the infimal displacement vector $v$. The first contribution of our work is to prove that this vector recovers certificates of primal and dual infeasibility whenever they exist. Based on this fact, we propose a simple way to extract approximate infeasibility certificates from the iterates of PDHG. We study three different sequences that converge to the infimal displacement vector: the difference of iterates, the normalized iterates, and the normalized average. All of them are easy to compute, and thus the approach is suitable for large-scale problems. Our second contribution is to establish tight convergence rates for these sequences. We demonstrate that the normalized iterates and the normalized average achieve a convergence rate of $O\left(\frac{1}{k}\right)$, improving over the known rate of $O\left(\frac{1}{\sqrt{k}}\right)$. This rate is general and applies to any fixed-point iteration of a nonexpansive operator. Thus, it is a result of independent interest since it covers a broad family of algorithms, including, for example, ADMM, and can be applied settings beyond linear programming, such as quadratic and semidefinite programming.
     Further, in the case of linear programming we show that, under nondegeneracy assumptions,
    the iterates of PDHG identify the active set of an auxiliary feasible problem in finite time, which ensures that the difference of iterates exhibits eventual linear convergence to the infimal displacement vector.
\end{abstract}

\section{Introduction}

First-order methods (FOMs) have been extensively studied by the optimization community since the late 2000s, following a long period where interior-point methods dominated research in continuous optimization. FOMs, which use only gradient information, are appealing for their simplicity and low computational overhead, in particular when solving large-scale optimization problems that arise in machine learning and data science. These methods have matured in many aspects (see, e.g., the recent textbook by Beck~\cite{BeckBook}) and are known to be useful for obtaining moderately accurate solutions to convex and non-convex optimization problems in a reasonable amount of time.  Despite this progress, FOMs have made only modest inroads into linear programming (LP), a fundamental problem in mathematical optimization.

FOMs applied to LP provide relatively simple methods whose most expensive operations are matrix-vector multiplications with the (typically sparse) constraint matrix. Such matrix-vector products are amenable to scale efficiently given increasingly ubiquitous computing resources like multi-threaded CPUs, GPUs~\cite{SpmvGPUBenchmark2020}, or distributed clusters~\cite{EcksteinMatyasfalvi2018}. In contrast, interior-point and simplex-based methods that dominate current practice are limited in how they use available computing resources because they depend on matrix inversion. To mark this distinction, Nesterov~\cite{Nesterov2014} defines methods that use at most matrix-vector products as capable of handling \textit{large-scale} problems and methods that use matrix inversion as handling \textit{medium-scale} problems. In the context of LP, these definitions of scale perhaps belie the reliability and practical efficiency of interior-point and simplex methods, but nevertheless the contrast in the computing requirements of the algorithms is an important one. Even though such computational aspects are outside the scope of this paper, it is this practical potential to efficiently solve large-scale LP that motivates the theoretical developments in this work.

While FOMs are typically studied in more general settings, the underlying assumptions and convergence rates in these settings do not necessarily hold or may not be tight for the special case of LP. Of particular relevance to this work, theory for FOMs is often developed under the assumption that an optimal solution exists, whereas LP solvers need to be able to detect infeasibility (i.e., when no optimal solution exists) and compute corresponding certificates. Infeasibility detection and computation of certificates are an essential aspect of solving LP, not only to provide feedback on modeling errors but also for algorithms that directly exploit LP certificates like Benders decomposition and branch-and-cut~\cite{Achterberg2007}.

\begin{algorithm}[t]
  \KwData{$x_0 \in \RR^d$}
  {\bf Step $k$:} ($k\geq 0$)\\
  $\qquad$ Update $\displaystyle x^{k+1} \leftarrow \prox_{\pssize f} (x^k - \pssize A^\top y^k)$,\\
  $\qquad$ Update $\displaystyle y^{k+1} \leftarrow \prox_{\dssize h} \left(y^k + \dssize A (2x^{k+1} - x^k)\right)$.
  \caption{Primal-dual hybrid gradient}
  \label{alg:pdhg}
\end{algorithm}
This work addresses the question of how to detect infeasibility in LP using the \emph{Primal-Dual Hybrid Gradient method} (PDHG). PDHG is a popular first-order method introduced by Chambolle and Pock \cite{chambolle2011first} to solve \emph{convex-concave minimax problems}, that is, problems of the form
\begin{equation}
  \label{eq:saddle}
  \min_{x \in \RR^n} \max_{y \in \RR^m}\ \dotp{Ax}{y} + g(x) - h(y)
\end{equation}
where $g: \RR^n \rightarrow \RR \cup \{\infty \}$ and $h: \RR^m \rightarrow \RR\cup \{\infty\}$ are proper lower semicontinuous convex functions and $A \in \RR^{m \times n}$. LP can be recast as a minimax problem through duality, and hence PDHG is applicable. The method consists of alternating updates between the primal and dual variables, see Algorithm~\ref{alg:pdhg}. In particular, when instantiated for LP, these updates correspond to matrix-vector products and projections onto simple sets (such as the positive orthant). In contrast with other methods like the \emph{Alternating Direction Method of Multipliers} (ADMM), PDHG does not require projections onto linear subspaces, which involve matrix inversions by direct or indirect methods.

The behavior of PDHG for \emph{feasible} problems (i.e., problems that have an optimal solution) has been studied in depth under several regularity assumptions. In their seminal work, Chambolle and Pock~\cite{chambolle2011first} show that the algorithm converges at a rate of $O(1/k)$ given appropriate choices for the step sizes $\pssize$ and $\dssize$. However, the situation for infeasible problems remains largely unstudied.

While it is relatively straightforward to formulate always-feasible auxiliary problems that can be used to detect infeasibility, for example, by penalizing violations of primal and dual constraints, this approach is unappealing for two reasons: First is the aesthetic interest of having a single algorithm that robustly handles all possible input~\cite{lamperski2020oblivious}. Second is the practical interest in effectively using available computing resources, as solving such auxiliary problems would approximately double the necessary work. Instead, we aim to use \emph{one} execution of PDHG and ask the following question:
\begin{quote}
  \begin{center}
  \it
  Do the PDHG iterates encode information about infeasibility?
  \end{center}
\end{quote}
We answer this question in the affirmative. We show that if the primal (and/or dual) problem is infeasible, the iterates of PDHG recover primal (and/or dual) infeasibility certificates. Moreover, we completely characterize the behavior of the iterates under different infeasibility settings. Before diving into our main contributions, let us present an illustrative example.
Recall that for a primal-dual LP pair, there exist three exhaustive and mutually exclusive possibilities: (1) both primal and dual are feasible, (2) both primal and dual are infeasible, and (3) one of the two problems is unbounded, and consequently, the other problem is infeasible. Small numerical experiments reveal that the behavior of PDHG is different depending on the setting.

\begin{example}\label{ex:toys}
Consider the LP problem with constants $\alpha, \beta \in \RR$:
\begin{align*}
\begin{split}
    \text{minimize}  &\quad x_0 + x_1 - \alpha x_2\\
    \text{subject to} &\quad x_0 + 2x_1  \leq 2 \\
    &\quad 3x_0 + x_1 \leq 2 \\
    &\quad x_0 + x_1 \geq \beta\ .
    \end{split}
\end{align*}
Figure \ref{fig:examples} displays three choices of $\alpha$ and $\beta$
\begin{enumerate}
\item \textbf{Both feasible}. Set $\alpha = 0$ and $\beta = 1$, then both primal and dual problems are feasible. In this case, both the primal and dual variables converge to a solution.
\item \textbf{Both infeasible}. Set $\alpha = 1$ and $\beta = 2$, then both primal and dual are infeasible. We observe that both primal and dual iterates diverge at a rate proportional to the number of iterations.
\item \textbf{Unbounded dual}. Set $\alpha = 0$ and $\beta = 2$, then the dual problem is unbounded and, thus, the primal problem is infeasible and the dual is feasible. Then the dual iterates diverge, and, interestingly, the primal iterates converge.
\item \textbf{Unbounded primal}. Set $\alpha = 1$ and $\beta = 1$, then the primal problem is unbounded and, thus, the dual problem is infeasible and the primal is feasible. Then the primal iterates diverge, and the dual iterates converge.
\end{enumerate}
\begin{figure}
    \centering
   \begin{subfigure}[b]{0.24\textwidth}
        \includegraphics[width=\textwidth]{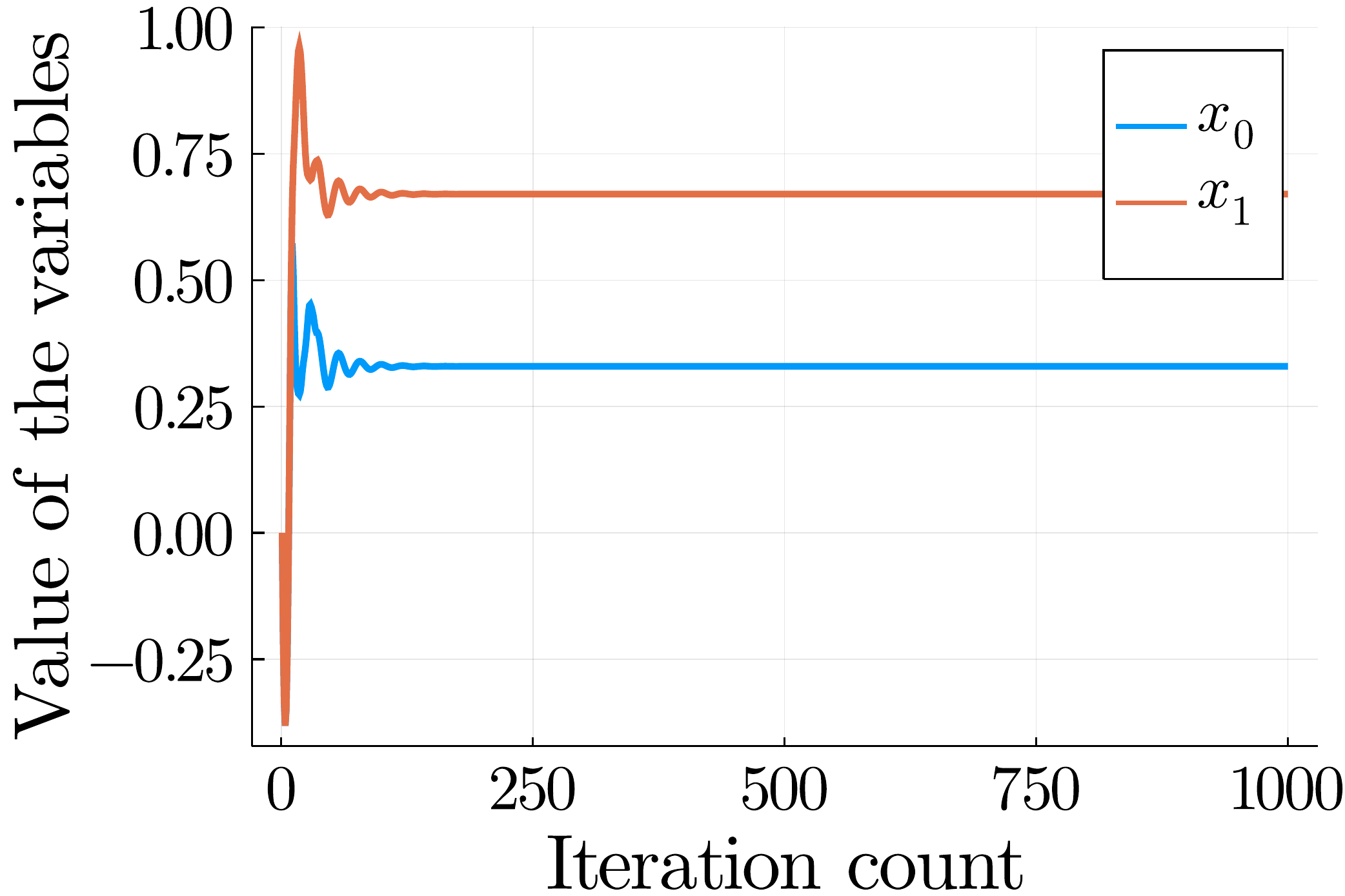}
            \end{subfigure}
    \begin{subfigure}[b]{0.24\textwidth}
      \includegraphics[width=\textwidth]{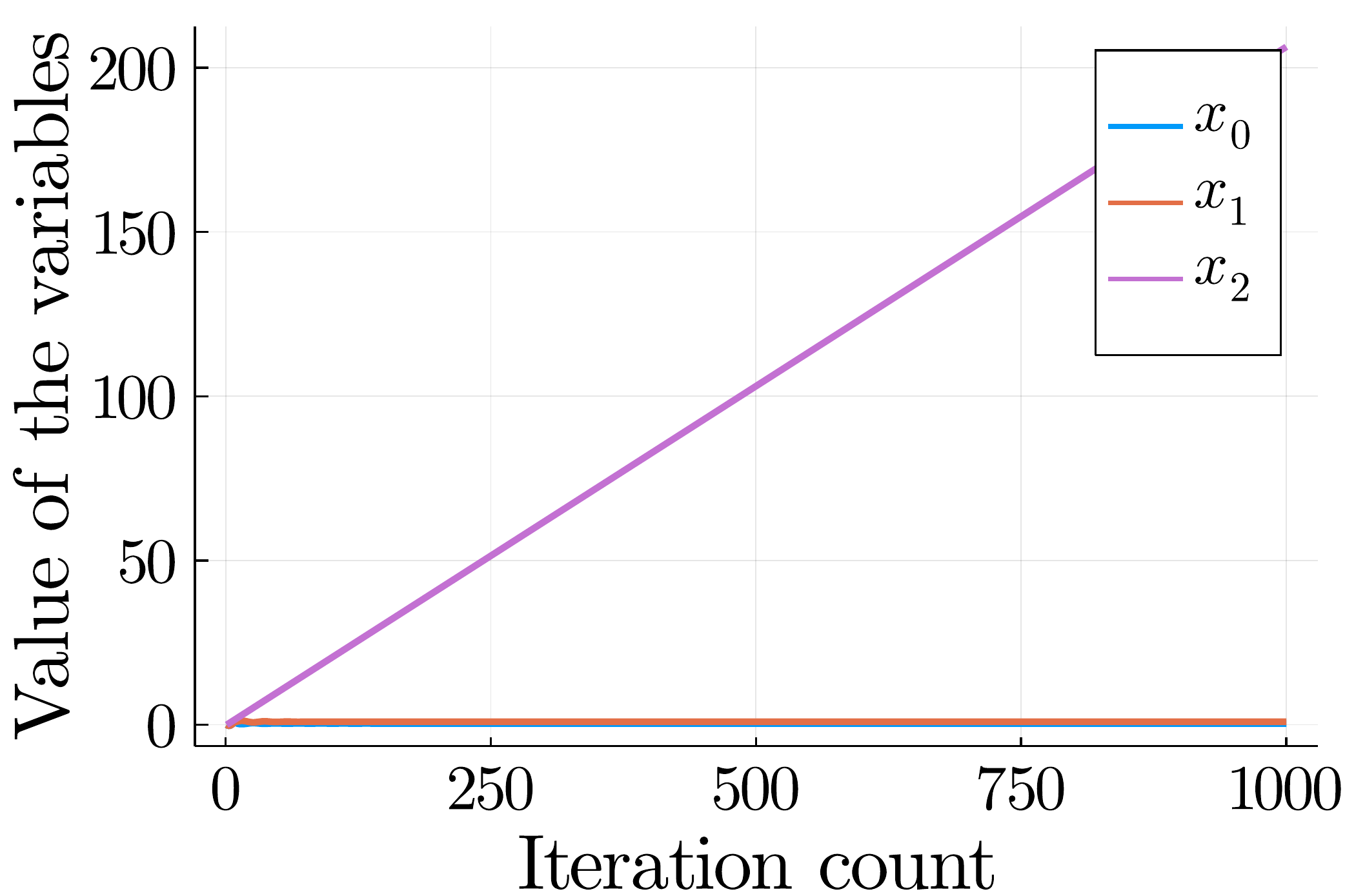}
    \end{subfigure}
      \begin{subfigure}[b]{0.24\textwidth}
        \includegraphics[width=\textwidth]{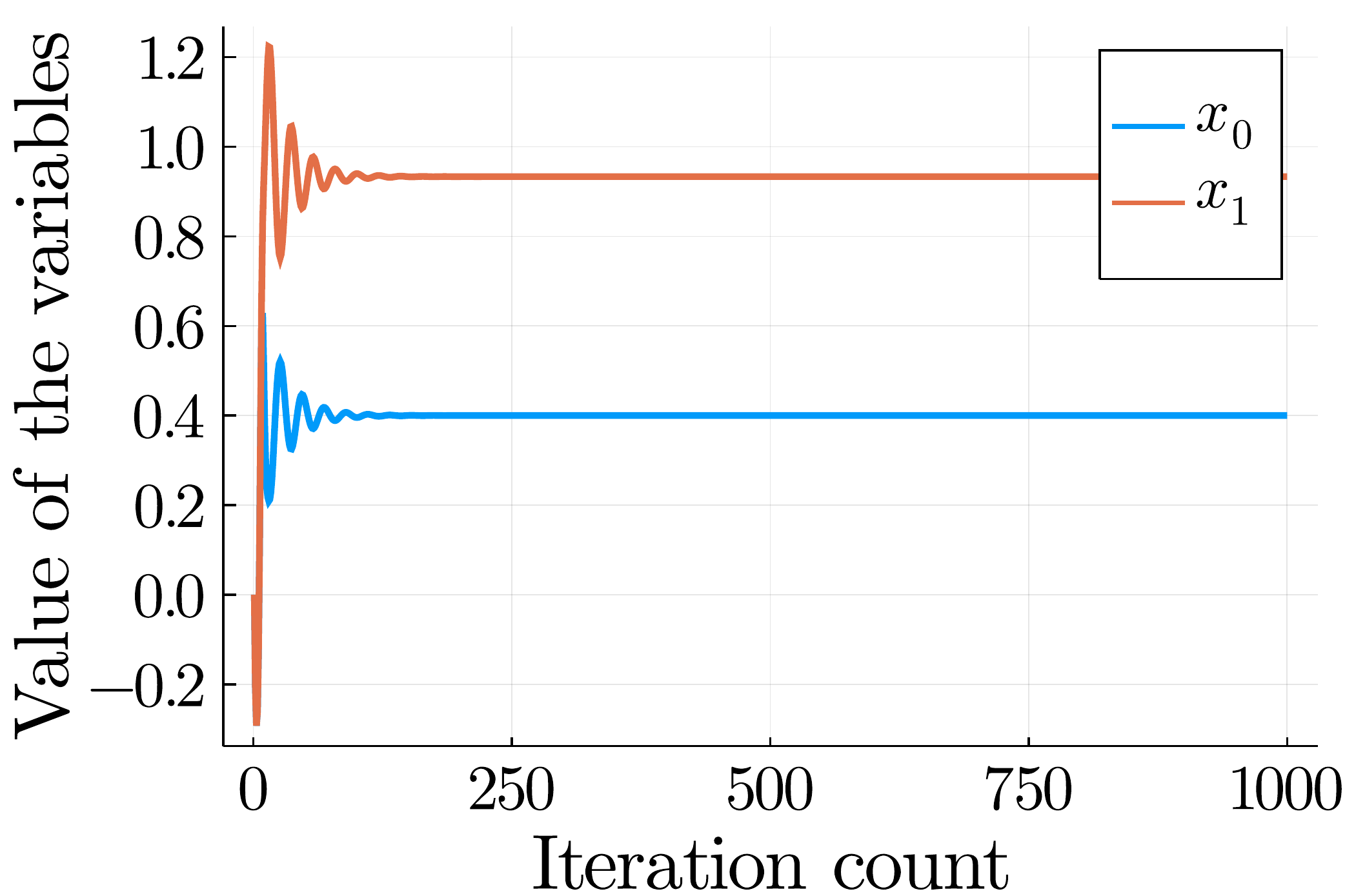}
    \end{subfigure}
       \begin{subfigure}[b]{0.24\textwidth}
        \includegraphics[width=\textwidth]{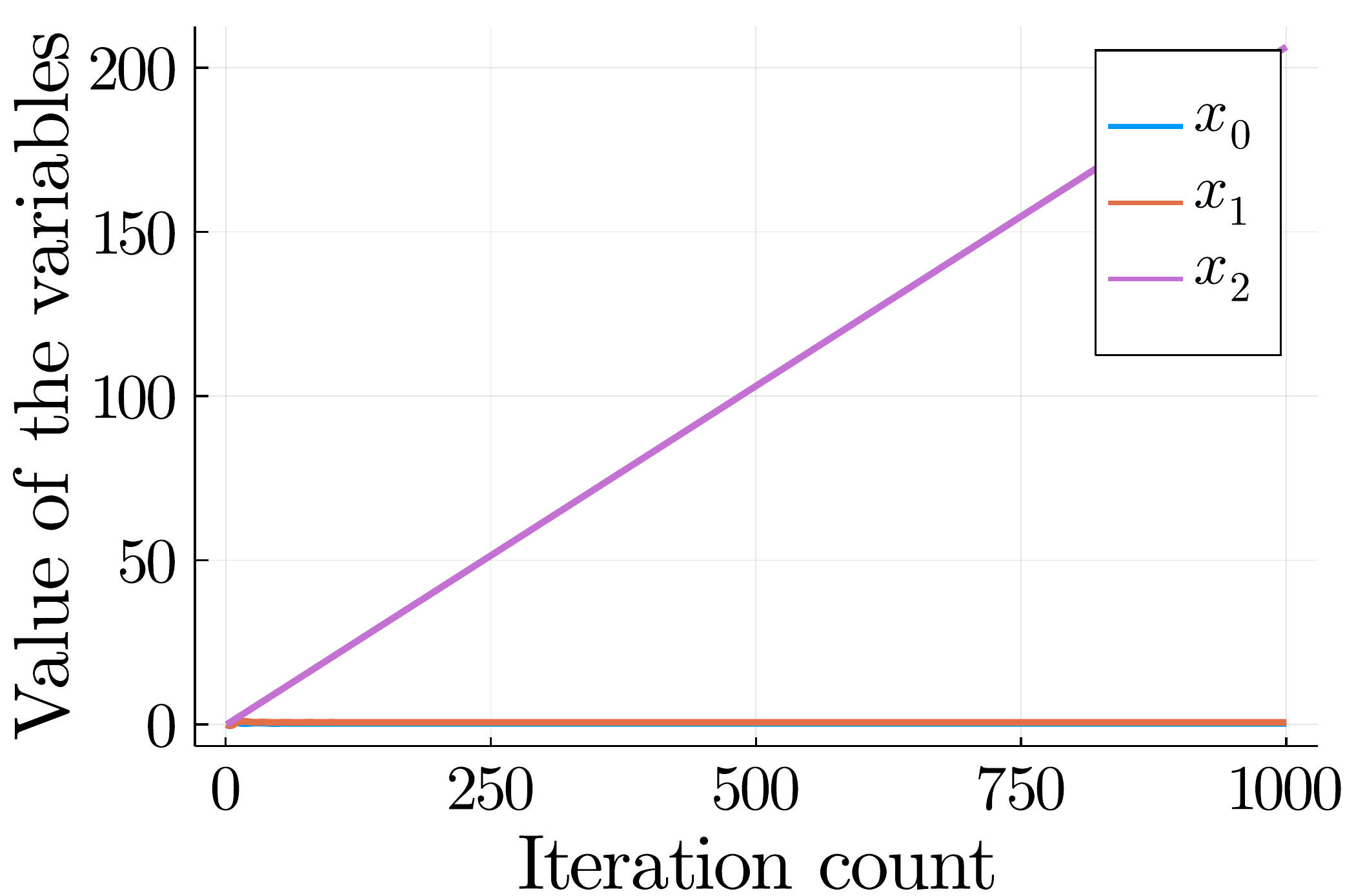}
    \end{subfigure}
    \begin{subfigure}[b]{0.24\textwidth}
        \includegraphics[width=\textwidth]{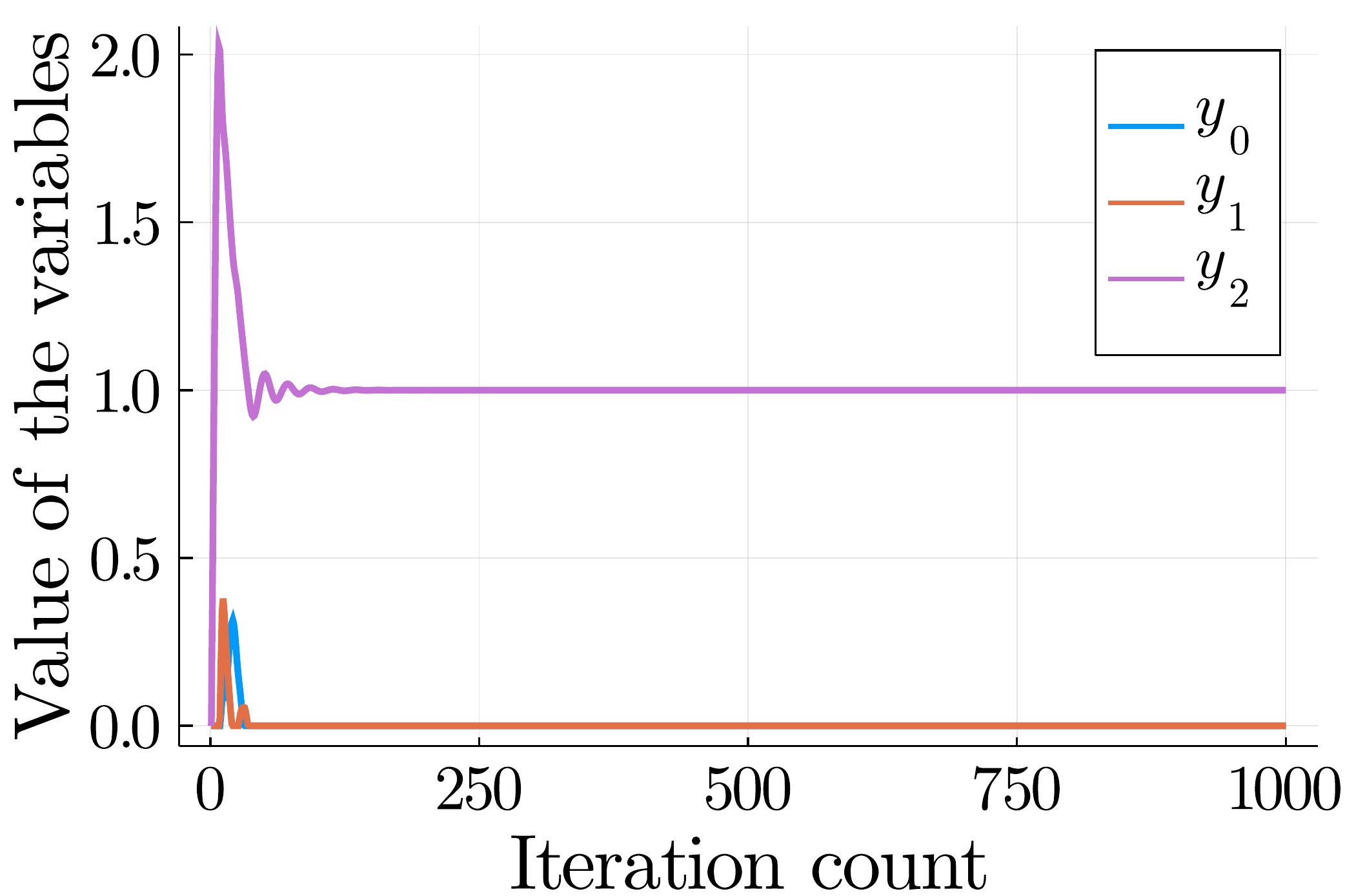}
        \caption{Both feasible}
        \end{subfigure}
    \begin{subfigure}[b]{0.24\textwidth}
        \includegraphics[width=\textwidth]{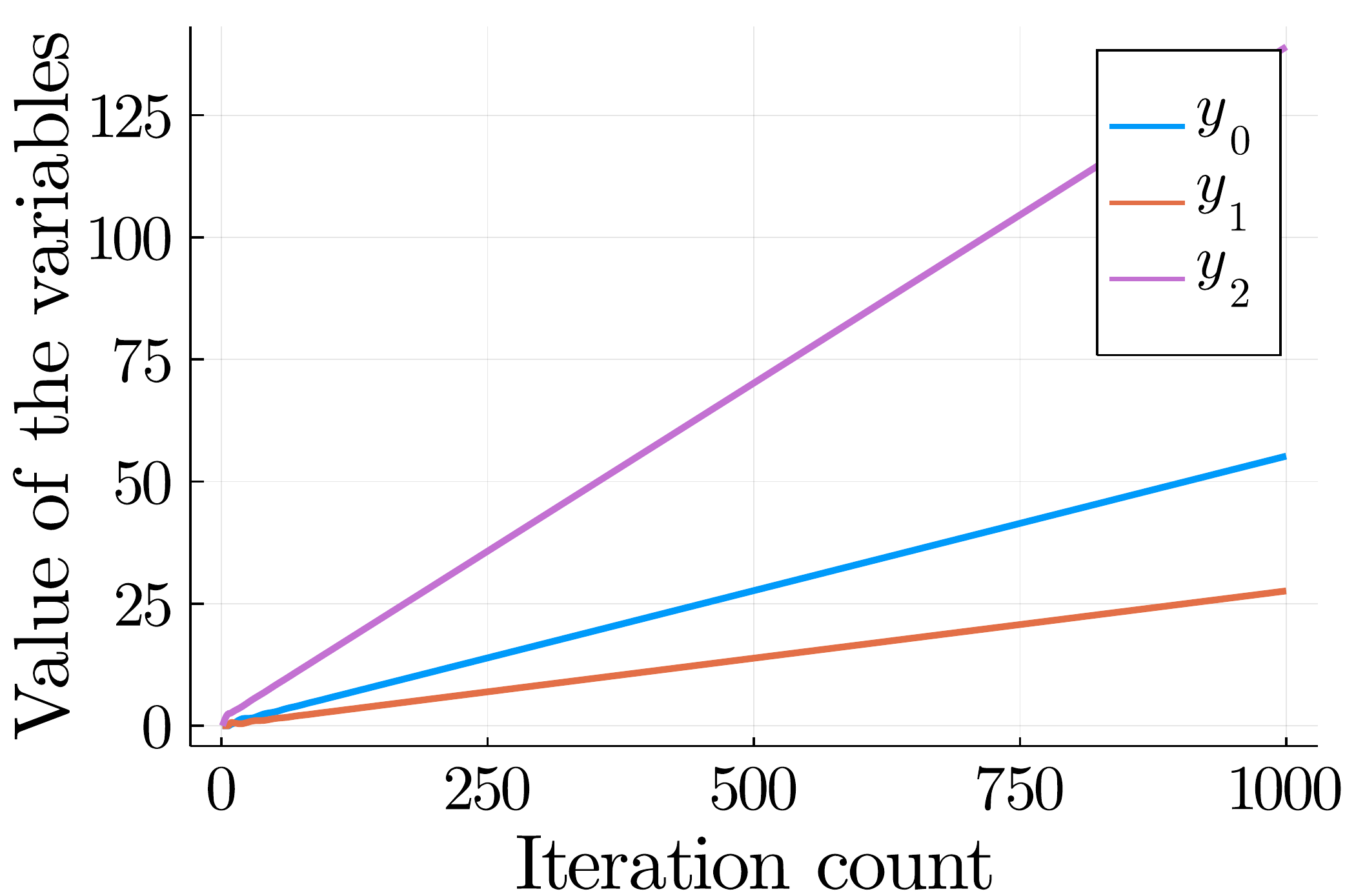}
        \caption{Both infeasible}
    \end{subfigure}
    \begin{subfigure}[b]{0.24\textwidth}
        \includegraphics[width=\textwidth]{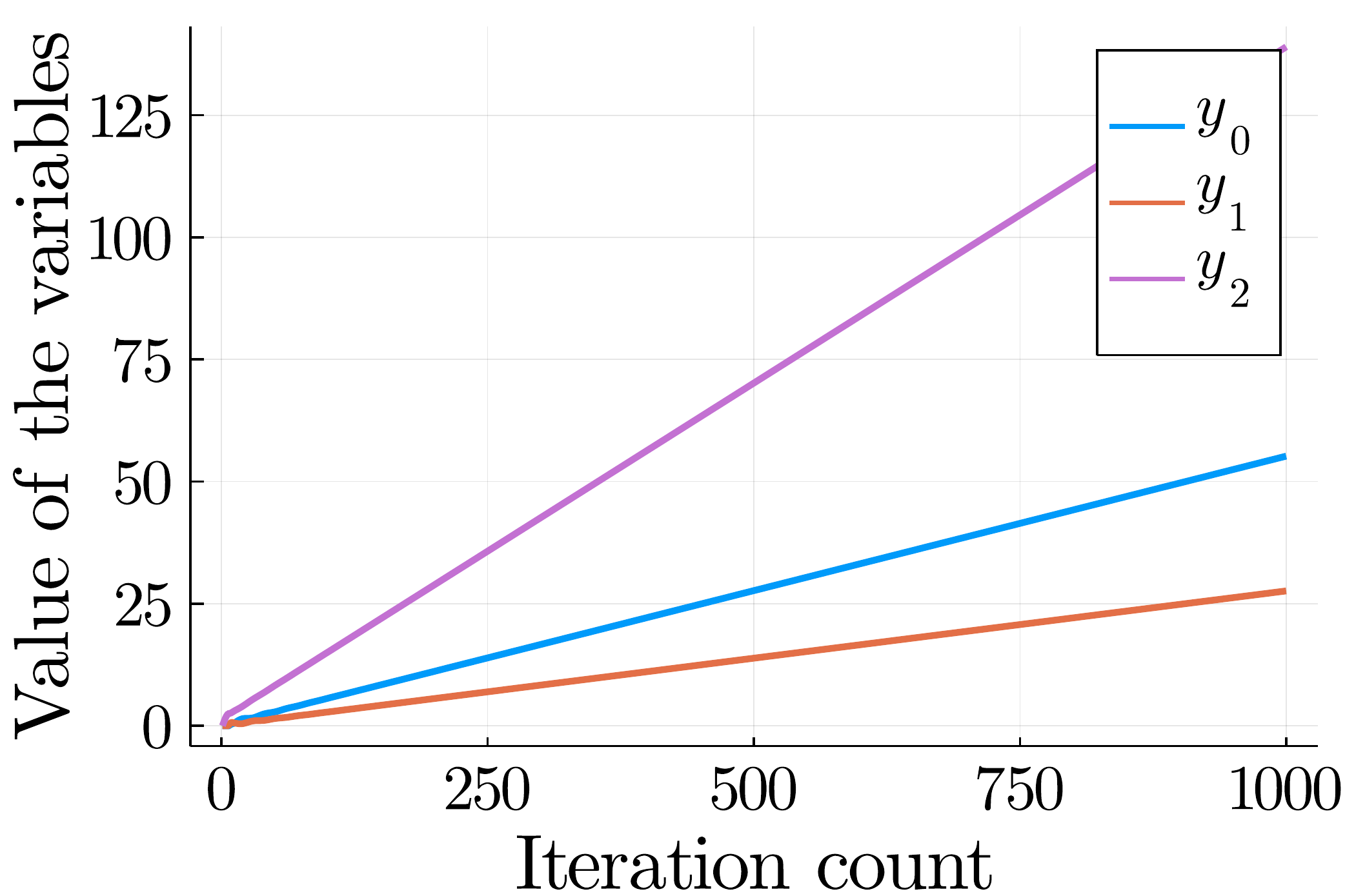}
        \caption{Unbounded dual}
    \end{subfigure}
        \begin{subfigure}[b]{0.24\textwidth}
        \includegraphics[width=\textwidth]{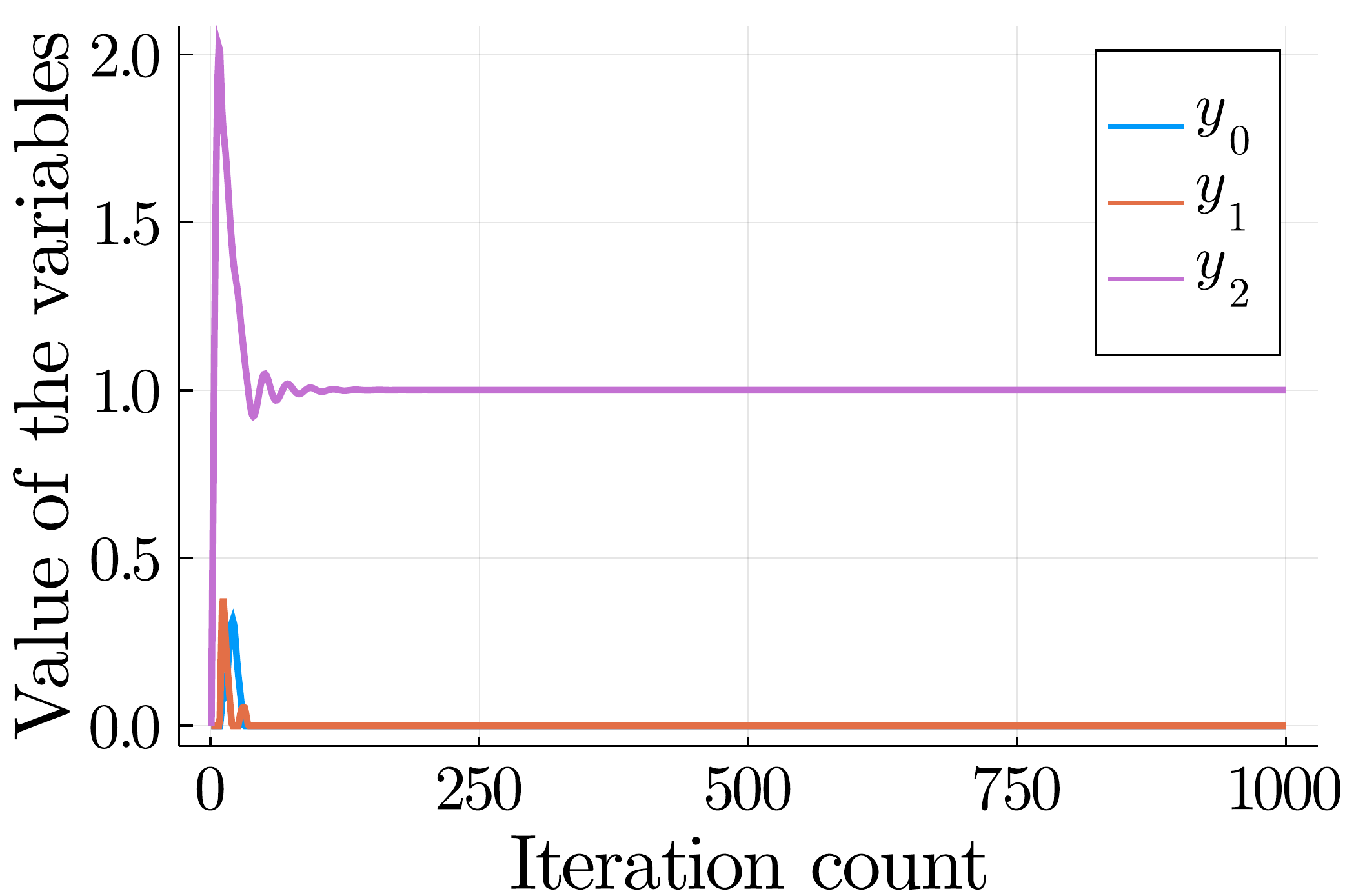}
        \caption{Unbounded primal}
    \end{subfigure}
    \caption{Four different settings depicted in Example~\ref{ex:toys}. Every subplot shows the component-wise value of the iterates against the iteration count. The first and the second rows correspond to the primal and dual iterates, respectively.}\label{fig:examples}
\end{figure}
\end{example}
From the experiments, we see that the iterates have a very stable asymptotic behavior. In particular, if the primal is feasible, then the dual variables converge, and analogously if the dual is feasible, then the primal iterates converge. Similarly, whenever the primal is infeasible, the dual iterates diverge at a controlled linear rate and vice-versa. Such behavior has not been previously observed or characterized in the literature.
\subsection{Main contributions}

For notational convenience, we use $z=(x, y)$ as the primal-dual pair, and $\bar{z}^k:=\frac{1}{k}\sum_{j=1}^k z^j$ as the average of iterates. We propose to detect infeasibility using three sequences:
\begin{subequations}\label{eq:sequences}
\begin{align}
        \textbf{(Difference of iterates)}~~~& d_k = (z^{k+1}-z^k)\ , \label{eq:diff_iterates} \\
        \textbf{(Normalized iterates)}~~~& \frac{z^k}{k}\ , \label{eq:normalized_iterates} \\
        \textbf{(Normalized average iterates)}~~~& \frac{2}{k+1}\bar{z}^k \label{eq:normalized_average}\ .
\end{align}
\end{subequations}

 Our proposal to detect infeasibility is as follows:
\begin{quote}
    Use these three sequences' primal and dual components as candidates for dual and primal infeasibility certificates. The algorithm should periodically check if any of these iterates satisfy the conditions that define an infeasibility certificate within numerical tolerances. If at any point this happens, it should conclude that the problem is (primal or dual) infeasible.
\end{quote} The overhead cost of extracting the certificates is negligible, making it suitable for large-scale problems. Most of the content of this work justifies this strategy theoretically.

Operator theory shows that all three of these sequences converge to a point $v$ known as the \emph{infimal displacement vector}. Section \ref{sec:preliminaries} will give a formal definition of this and other relevant concepts. We list our contributions assuming, for now, the existence of such a vector $v$.

\textbf{Sublinear convergence rate to the infimal displacement vector} (Theorem~\ref{thm:general-convergence}). It is natural to wonder how fast the three sequences \eqref{eq:sequences} converge to the infimal displacement vector. We study this question through the lenses of operators and fixed-point iteration. To the best of our knowledge, the only known result in this vein ensures a rate of $O\left(\frac {1}{\sqrt{k}}\right)$ for the difference of iterates $d^k$ \cite{liu2019new}, which is known to be tight \cite{davis2016convergence}.
In this same situation, we show that both the normalized iterates and the normalized average iterates converge at a faster rate of $O\left(\frac{1}{k}\right)$. This result generalizes to \textit{any fixed-point iteration of a nonexpansive operator}, not only the firmly nonexpansive operators studied in \cite{liu2019new}. Specifically, it also applies to many popular first-order methods, including but not limited to PDHG, ADMM \cite{powell1978algorithms}, and Mirror-prox \cite{nemirovski2004prox}, and it extends to other settings beyond LP such as quadratic convex programming and semidefinite programming. Furthermore, we show that this result is tight for PDHG; i.e., there exist instances with a convergence rate lower bounded by $\Omega\left(\frac{1}{k}\right).$ This result suggests that current ADMM-based codes like OSQP \cite{osqp} that use exclusively the difference of iterates to detect infeasibility should additionally consider the normalized iterates and normalized average iterates.

\textbf{Characterization of the iterates for infeasible problems} (Theorem~\ref{thm:characterization}).  We characterize the behavior of PDHG for all the LP feasibility scenarios (see Table \ref{tab:behaviors}). In particular, we show that if the primal (or dual) iterates diverge, then the iterates diverge in the direction of a ray, where the direction of the ray recovers certificates of dual (or primal) infeasibility. Such direction turns out to be the infimal displacement vector $(v_x, v_y)$. This justifies using the sequences \eqref{eq:sequences} as infeasibility certificate candidates.
Furthermore, we show that when the primal problem is feasible, then the dual iterates, without any normalization, converge to some $y^\star$ that is closely related to $v$. An analogous result holds if the dual is feasible. This describes the dynamics of PDHG for unbounded problems. The next table summarizes our findings for the four possible cases.

\begin{table}[h]
\centering
\begin{tabular}{|c|c|c|}
\hline
\diagbox{\textbf{Primal}}{\textbf{Dual}} & \textbf{Feasible} & \textbf{Infeasible}  \\
\hline
\textbf{Feasible}                        &    $x^k,y^k$ both converge               &           $x^k$ diverges, $y^k$ converges           \\
\hline
\textbf{Infeasible}                      &     $x^k$ converges, $y^k$ diverges              &       $x^k,y^k$ both diverge               \\
\hline
\end{tabular}\label{tab:char}
\caption{Bahavior of PDHG for solving under different feasibility assumptions. }
\label{tab:behaviors}
\end{table}

\textbf{Eventual linear convergence for nondegenerate problems} (Theorem~\ref{thm:identifiability}). In the process of characterizing the dynamics of PDHG, we show that the iterates $(x^k, y^k)$ always converge to a unique ray $\{(x^\star, y^\star) + \lambda v \mid \lambda \in \RR\}$. We show that if the limit point $(x^\star, y^\star)$ satisfies a strict complementarity condition, for a related auxiliary problem, then the iterates $(x^k, y^k)$ fix their active set after finitely many iterations. In turn, this leads to the eventual linear convergence of the difference of iterates \eqref{eq:diff_iterates}. Formally, we show that there exists $K \geq 0$ such that for all sufficiently large $k \geq K$ we have
$$ \|d^k - v\| \leq O\left(\gamma^{k-K}\right) \qquad \text{for some }\gamma \in (0, 1)\ .$$
We further show that even after the active set is fixed, the normalized iterates and normalized average do not exhibit faster convergence. Thus, it is strictly better to use the difference of iterates to detect infeasibility in this regime.

\textbf{Computational experiments} (Section \ref{sec:experiments}). We complement our theoretical results with numerical experiments displaying the efficacy of the different certificate candidates \eqref{eq:sequences}. Specifically, our experiments show that using all three sequences in \eqref{eq:sequences} is beneficial. On the one hand, if the active set's finite time identification occurs relatively quickly, then the differences of iterates \eqref{eq:diff_iterates} exhibit faster convergence. On the other hand, for some problems, identifying the active set might not happen in a reasonable amount of time. In this case, both the normalized iterates \eqref{eq:normalized_iterates}  recovers approximate infeasibility certificates much more efficiently than the differences.

\subsection{Related work}

\textbf{Primal-dual hybrid gradient.}
Chambolle and Pock~\cite{Chambolle2016} review PDHG among other methods and describe its applications in computer vision. O'Connor and Vandenberghe~\cite{OConnor2020} show that PDHG is in fact a particular application of Douglas-Rachford Splitting (DRS)~\cite{douglas1956numerical, glowinski1975approximation, gabay1976dual, lions1979splitting}.

\textbf{FOMs for LP.} Lan et al.~\cite{Lan2011} and Renegar~\cite{Renegar2019} develop FOMs for LP, considered as a special case of semidefinite programming, with $O\left(\frac{1}{k}\right)$ convergence rates. Gilpin et al.~\cite{Gilpin2012} obtain a restarted FOM for LP with a linear convergence rate. These analyses assume an optimal solution exists. Pock and Chambolle~\cite{PockChambolle2011} apply PDHG with diagonal preconditioning to LP on a small number of test instances. They note that on small-scale problems, interior-point methods clearly dominate, while their method outperforms MATLAB's LP solver on one larger LP motivated by a computer vision application. Most recently, Basu et al.~\cite{pmlr-v119-basu20a} apply accelerated gradient descent to smoothed LPs, obtaining solutions to industrial problems with up to $10^{12}$ variables.

\textbf{Infeasibility detection.}
Classically, the primal simplex method for LP detects primal infeasibility while solving a ``phase-one'' auxiliary problem for an initial feasible basis and detects dual infeasibility based on conditions when computing a step size (i.e., the ratio test)~\cite{maros2003}. Infeasibility certificates are extracted from the iterates of interior-point methods without substantial extra work~\cite{Todd2004}. Infeasibility detection is only the first step of diagnosing the cause of the infeasibility in an LP model~\cite{Chinneck1996}.

Most research on infeasibility detection capabilities for FOMs for convex optimization has focused on ADMM or equivalently Douglas-Rachford Splitting. Eckstein and Bertsekas~\cite{eckstein1992douglas} show that when no solution exists, then the iterates diverge. Recent practical successes motivated further research in this direction, characterizing the asymptotic behavior of the iterates under additional assumptions. For example, the line of work \cite{bauschke2004finding, bauschke2016douglas, moursi2016douglas} studies Douglas-Rachford applied to problems that look for a point at the intersection of two non-intersecting convex sets. On the other hand, Raghunathan and Di Cairano~\cite{raghunathan2014infeasibility} investigate the asymptotic dynamics of ADMM for convex quadratic problems when the matrices involved are full rank.

Banjak et al.~\cite{banjac2019infeasibility} show that the infimal displacement vector of ADMM recovers certificates of infeasibility for convex quadratic problems with conic constraints. Based on this, they proposed to use the difference of iterates to test infeasible. Complementing this work, \cite{liu2019new} establishes a $O\left(\frac{1}{\sqrt{k}}\right)$ convergence rate for the difference of iterates of any algorithm that induces a firmly nonexpansive operator and introduced a scheme that utilizes multiple runs of ADMM to detect infeasibility. This type of scheme aims to handle pathological scenarios that do not occur in LP.

O'Donoghue et al.~\cite{o2016conic} propose to apply ADMM to a homogeneous self-dual embedding of a convex conic problem\footnote{Linear objective with conic constraints.}. A nice byproduct of this approach is that it automatically produces infeasibility certificates. Subsequent work~\cite{o2020operator} extends this approach to Linear Complementarity Problems, which cover quadratic convex losses with conic constraints.

To our knowledge, the only work analyzing the behavior of PDHG on potentially infeasible instances is by Malitsky~\cite{malitsky2019}, which considers linearly constrained problems. This analysis applies only to linear equality constraints, not to linear inequalites present in LP.

\textbf{Finite time identifiability.} Finite time identifiability has a long history in the field of optimization.
This phenomenon is first documented for the projected gradient descent method \cite{Dunn87, Calamai-More87, Burke-More88, Burke90}. Soon after it is studied for other methods, such as the Proximal Point Method \cite{Ferris91} and Projected subgradient descent \cite{Flam92}, among others \cite{Al-Khayyal-Kyparisis91}. Identifiability is also exploited as tool for algorithmic design for the so called ``$\mathcal{UV}$-algorithms'' \cite{MC05}. Recent works \cite{molinari2019convergence, liang2017local, liang2018local} study finite time identification for popular FOMs. In particular, Liang et al.~\cite{liang2018local} show that the iterates of PDHG identify the active constraints in finite time, provided the limit point is nondegenerate. All of these works assume the underlying problem is feasible. The significant number of algorithms exhibiting this behavior motivated researchers to develop general theory (even beyond the realms of optimization) \cite{Wright, MC03, MC04, lewis_active, drusvyatskiy2014optimality, lewis2018partial}.  We refer the interested reader to \cite{lewis2018partial} for an elegant geometrical definition that generalizes most notions of nondegeneracy.

\subsection{Outline} The paper is structured as follows. Section \ref{sec:preliminaries} presents all the necessary background. In Section \ref{sec:sublinear-rate}, we show a convergence rate of $O(1/k)$ for the normalized iterates and normalized average generated by the fixed-point iteration of a nonexpansive operator. Then, Section \ref{sec:characterization} shows a complete characterization of the behavior of PDHG under different infeasibility assumptions. In Section \ref{sec:identifiability}, we study a condition that ensures finite time identifiability of the active set. We show that under this condition, the difference of iterates exhibits eventual linear convergence. We present numerical experiments that complement the theoretical results in Section \ref{sec:experiments}. We close the paper with some conclusions and open questions in Section \ref{sec:conclusions}.

\subsection{Notation}
We use the symbols $\NN$ and $\RR$ to denote the natural numbers and reals. Our results take place in a finite dimensional spaces $\RR^d$. We denote the cone of nonnegative vectors as $\RR^d_+$. We endow $\RR^d$ with the standard dot product $\dotp{x}{y} = x^\top y$ and norm $\|x\|^2_2 := \dotp{x}{x}$. We use $I$ to denote both the identity map and the identity matrix in $\RR^{d\times d}$.  Given a matrix $M \in \RR^{d \times d}$ we use $\lambda_1(M), \dots, \lambda_d(M)$ to denote its eigenvalues. We define the spectral radius of a matrix $M$ as $\rho(M) = \max_{j \in [d]} |\lambda_j(M)|$. Similarly, for a matrix $M\in \RR^{m \times d}$, with $d\leq m$, we use $\sigma_1(M) \geq \dots \geq \sigma_d(M)$ for its singular values and we let $\sigma_{\max}(M)$ and $\sigma_{\min}(M)$ denote the maximum and minimum nonzero singular values. We define the operator norm as $\|M\|_2 = \sigma_{\max}(M).$ For any matrix $M$ we denote its pseudo-inverse as $M^\dagger$. A symmetric matrix $M \in \RR^{d \times d}$ is positive definite if its minimum eigenvalue is positive, $\min_{j \in [d]}\lambda_j(M) > 0.$ Every positive definite matrix $M$ defines an inner product and norm given by $\dotp{x}{y}_M = x^\top M y$ and $\|x\|_M^2 = \dotp{x}{x}_M$, respectively. Frequently, we will make use of the symbol $\|\cdot\|$ to denote an arbitrary matrix norm in $\RR^n.$ Often, $\|\cdot\|$ will denote a norm with respect to which an operator is (firmly) nonexpansive; see next section for a formal definition.

Let $C \subseteq \RR^d$ be an arbitrary set. The symbol $\cl(C)$ denotes the closure of $C$. We use $\proj_C(\cdot)$ to denote the Euclidean projection mapping onto $C$; this map is well-defined provided $C$ is closed and convex. We use $(z)_+$ as a shorthand for $\proj_{\RR^d_+}(z).$ We define the indicator function of $C$ as  $\iota_C:\RR^d \rightarrow \RR \cup \{\infty\}$ given by $\iota_C(z) = 0$ if $z \in C$ and $\iota_C(z) = \infty$ otherwise. Given a map $T:\RR^d \rightarrow \RR^d$, its range is defined as $\range(T) = \{ T(z) \mid z\in \RR^d\}$. Given two mappings $T_1, T_2:\RR^d \rightarrow \RR^d$, we denote their composition as $T_1 \circ T_2$, that is $T_1 \circ T_2 (z) = T_1(T_2(z))$. Since we study primal and dual problems, we use $z = (x, y) \in \RR^{n+m}$ as a placeholder for primal and dual variables. We will sometimes refer to a vector $v \in \RR^{n + m}$ and use $v_x$ and $v_y$ to denote its primal and dual components. We use superscripts to denote iteration counts, consequently $z^k$ is the $k$th iterate. We use $\supp(x) := \{i \in [d] \mid x_i  \neq 0\}$ to denote the support of the vector $x$.

\section{Preliminaries}\label{sec:preliminaries}

\textbf{Convex analysis} \cite{con_ter, Borwein-Lewis, BC_book}. For any closed convex function $f: \RR^d \rightarrow \RR \cup \{\infty\}$, we define the subdifferential set of $f$ at $\bar x \in \RR^m$, denoted by $\partial f(\bar x)$, as the set of all $g \in \RR^d$ such that
\begin{equation*}
    f(x) \geq f(\bar x) + \dotp{g}{x - \bar x} \qquad \text{for all }x \in \RR^d\ .
\end{equation*}
This is a generalization of the gradient mapping for functions that lack differentiability. The subdifferential set characterizes the set of global minima. In fact, it is easy to show that $x^\star$ is a minimizer of $f$ if, and only if, $0 \in \partial f (x^\star).$
For any closed convex set $C \subseteq \RR^d$, we define its \emph{normal cone} at $\bar x \in C$ to be
\begin{equation*}
    N_C(\bar x) : = \{g \in \RR^d \mid \dotp{g}{x - \bar x} \leq 0 \quad \text{for all }x\in C\}\ .
\end{equation*}
Analogously, $N_C(\bar x)$ can be seen as the set of all points $x$ such that $\proj_C(x) = \bar x$.
The convex subgradient is known to follow a nice of set of calculus rules. In this work we will make use of two of them: let $f_1: \RR^d \rightarrow \RR$ and $f_2: \RR^d \rightarrow \RR \cup \{\infty\}$ be closed convex functions and $C$ be a closed convex set, then
\begin{equation}\label{eq:calculus}
    \partial (f_1 + f_2) (x) = \partial f_1(x) + \partial f_2(x) \qquad \text{and} \qquad \partial \iota_C(x) = N_{C}(x)  \qquad \text{for all }x\in\RR^d\ .
\end{equation}
Given any convex function $f$, we define its \emph{proximal operator} $\prox_f: \RR^d \rightarrow \RR^d$ as the unique solution of
\begin{equation}
    \prox_f (\bar x) := \argmin_{x} f(x) + \frac{1}{2}\|x - \bar x\|^2_2\ .
\end{equation}
Using optimality conditions one can prove that
\begin{equation}\label{eq:re-prox}
    x \in (I + \partial f ) (x^+) \iff x^+ = \prox_f (x)\ .
\end{equation}

\textbf{Linear programming} \cite{Dantzig63,Martin1999}. LP problems can be parameterized using multiple equivalent forms. For our theoretical results we focus on the \emph{standard form} of LP:
\begin{align}\tag{\textbf{P}}\begin{split}\label{primal}
    \text{minimize}  &\quad c^\top x\\
    \text{subject to} &\quad Ax  = b \,\,\, (\in \RR^m)\\
    &\quad x \geq 0 \, \quad (\in \RR^n)\ ,
    \end{split}
\end{align}
where $A \in \RR^{m \times n}$, $b \in \RR^m$, and $c \in \RR^n$ are given. The dual of this problem is given by
\begin{align}\tag{\textbf{D}}\begin{split}\label{dual}
    \text{maximize}  &\quad b^\top x\\
    \text{subject to} &\quad A^\top y  \leq c \,\,\, (\in \RR^n)\ .
    \end{split}
\end{align}
Although the proofs in this paper are tailored to this form, the techniques we use should extend easily to any other form.

Farkas' Lemma states that a feasible solution of \eqref{primal} exists if, and only if, the following set is empty
\begin{equation}\label{eq:dual-cert}
    \{y \in \RR^m \mid b^\top y < 0 \text{ and } A^\top y \geq 0\}\ .
\end{equation}
We call the elements of this set \emph{certificates of primal infeasibility}, as their existence guarantees that the primal problem is infeasible. Analogously, the certificates of infeasibility for the dual problem \eqref{dual} are
\begin{equation}\label{eq:primal-cert}
    \{x \in \RR^n \mid c^\top x < 0, Ax = 0 \text{ and } x\geq 0\}\ .
\end{equation}

\textbf{Primal-dual hybrid gradient}. Chambolle and Pock \cite{chambolle2011first} establish convergence to a saddle point at a rate of $O(1/k)$ provided that a \emph{saddle exists} and $\pssize \dssize \|A\|_2^2 < 1$.\footnote{More precisely, Chambolle and Pock proved this rate for the primal-dual gap of the averaged iterates.} The primal-dual problems \eqref{primal}-\eqref{dual} can be recast as a \emph{convex-concave saddle point problem}. In particular we choose $g(x) = c^\top x + \iota_{\{x\geq 0\}}(x)$ and $h(y) = b^\top y$. In this case the proximal updates can be computed in closed form. In fact, a PDHG update reduces to
\begin{align}\begin{split}\label{eq:lpdhg}
    x^+ &= \proj_{\RR_+^n}(x- \pssize A^\top y - \pssize c)\\
    y^+ &= y + \dssize A(2x^+ -x) - \dssize b\ .
\end{split}
\end{align}
Observe that the most complex operations in the update formula are matrix-vector products, and all other operations are separable by component.

An update of PDHG (Algorithm~\ref{alg:pdhg}) can be equivalently defined with a differential inclusion of the form
\begin{equation}\label{eq:cp-characterization}
     M \begin{bmatrix} x^{k} - x^{k+1}\\ y^k - y^{k+1}\end{bmatrix}  \in \begin{bmatrix} \partial g(x^{k+1}) \\ \partial h(y^{k+1}) \end{bmatrix} +  \begin{bmatrix} A^\top y^{k+1} \\ -Ax^{k+1}\end{bmatrix} \qquad \text{with} \qquad M := \begin{bmatrix} \frac{1}{\pssize}I_n & -A^T \\ -A & \frac{1}{\dssize}I_m \end{bmatrix}\ ,
\end{equation}
this follows from \eqref{eq:re-prox}. We will later leverage this inclusion in our proofs.

\textbf{Operators and the fixed-point iteration.}
We will find it useful to think of iterative algorithms from an operator viewpoint. Given an arbitrary map $T: \RR^d \rightarrow \RR^d$, the corresponding fixed-point iteration is defined as
\begin{equation}\label{eq:fixed-point}
    z^{k+1} = T (z^{k})\ .
\end{equation}
Most first-order methods can be described in this form. The primal-dual hybrid gradient method can be encoded as $T(x, y) = (x^+, y^+)$ where the output pair is defined in \eqref{eq:lpdhg}. When looking at an algorithm from this perspective, we transform the problem of finding a solution of the optimization problem to that of finding a fixed-point of the operator, i.e., $z^\star = T(z^\star)$. This idea has proven fruitful for proving optimal converge rates for a variety of algorithms \cite{davis2016convergence}.

Here we make a minimal assumption that is sufficient to analyze PDHG in the infeasible case.
An operator $T$ is said to be \emph{nonexpansive} if it is $1$-Lipschitz continuous with respect to a matrix norm $\|\cdot\|$, meaning that for any $z_1, z_2 \in \RR^d$ we have
\begin{equation}
\|T(z_1) - T(z_2)\| \leq \|z_1 - z_2\|\ .
\end{equation}
Nonexpansiveness does not ensure the convergence of iterates in the feasible case. Yet a slightly stronger condition does. An operator $T$ is \emph{firmly nonexpansive} if it satisfies
\begin{equation*}
    \|T (z_1) - T(z_2)\|^2 \leq \|z_1 - z_2\|^2 - \|(T-I) (z_1) - (T-I)(z_2)\|^2 \qquad \text{for all } z_1, z_2 \in \RR^d\ .
\end{equation*}
Note that the norm here is not necessarily the Euclidean norm. All the results concerning (firmly) nonexpansiveness in this section and the next one are with respect to the norm in which these properties hold.
The following is a beautiful geometrical result proved by Pazy that defines a pivotal object in our studies.
\begin{lemma}[Lemma 4 in \cite{pazy1971asymptotic}]\label{lemma:convex-range}
Let $T$ be a nonexpansive operator, then the set $\cl(\range(T-I))$ is convex. Consequently, there exists a unique minimum norm vector in this set:
\begin{equation}
    v_T := \argmin_{z \in \cl(\range(T-I))} \frac{1}{2}\|z\|^2\ .
\end{equation}
\end{lemma}
This vector is known as the \emph{infimal displacement vector}. We drop the subscript $T$ and make the corresponding operator clear from the context. Intuitively, $v$ is the minimum size perturbation we should subtract from $T$ to ensure it has a fixed point.
\begin{theorem}[\cite{pazy1971asymptotic} and \cite{bailion1978asymptotic}]\label{thm:pazy}
Let $T$ be a nonexpansive operator and $(z^k)$ be a sequence generated by the fixed-point iteration \eqref{eq:fixed-point}. Then, we have
\begin{equation}\label{eq:conv_iterate}
    \lim_{k \rightarrow \infty } \frac{z^k}{k} = v\ .
\end{equation}
If further $T$ is firmly nonexpansive, then
\begin{equation}\label{eq:conv_difference}
    \lim_{k\rightarrow \infty } z^{k+1} - z^k = v\ .
\end{equation}
\end{theorem}
That is the normalized iterate converges to the infimal displacement vector when $T$ is nonexpansive and if $T$ is firmly nonexpansive the difference of iterates also converge.
One might wonder whether the the stronger condition is necessary. This turns out to be the case.

The following proposition shows two things: first, \eqref{eq:conv_difference} is provably stronger than \eqref{eq:conv_iterate} and second, the convergence of the iterates ensures converges of the normalized average. Recall from the previous section that we use $\bar z^k := \frac{1}{k}\sum_{j=1}^k z^j$ to denote the average.

The proof of this proposition is technical so it is deferred to Appendix \ref{sec:proof-prop-1}.
\begin{proposition}\label{prop:convergence-implications}
Let $(z^k)^\infty_{k=1} \subseteq \RR^d$ be an arbitrary sequence and let $v \in \RR^d$ be a fixed vector. Then the following implications hold:
\begin{enumerate}
    \item \textbf{Difference convergence implies normalized iterate convergence}. \begin{equation*}
        \lim_{k\rightarrow \infty} (z^{k+1} - z^k) = v \implies \lim_{k\rightarrow \infty} \frac{z^k}{k} = v\ .
    \end{equation*}
    \item \textbf{Normalized iterate convergence implies normalized average convergence}. \begin{equation*}
       \lim_{k\rightarrow \infty} \frac{z_{k}}{k} \rightarrow v \implies \lim_{k \rightarrow \infty} \frac{2\bar z^k}{(k+1)} \rightarrow v\ .
    \end{equation*}
\end{enumerate}
Moreover, these implications cannot be reversed as there exist simple counterexamples in $\RR$.
\end{proposition}

Naturally when concerned with practical algorithms one would like to have convergence rates for \eqref{eq:conv_iterate} and \eqref{eq:conv_difference}. As far as we know, the state-of-the-art result in this vein is due to Liu, Ryu, and Yin \cite{liu2019new}.

\begin{theorem}[\cite{liu2019new}]\label{wotao}
Let $T$ be a firmly nonexpansive operator and $(z^k)$ be a sequence generated by \eqref{eq:fixed-point}. Then, for any $\varepsilon > 0$, then there exists a point $z_\varepsilon$ such that
\begin{enumerate}
    \item[] \textbf{(Average iterate rate)}.
    \begin{equation*}
    \left\|v - \frac{2}{k+1}(\bar z^k - z^0)\right\| \leq \sqrt{\frac{2}{k+1}}\|z^0 - z_\varepsilon\| + \varepsilon\ ,
    \end{equation*}
    \item[] \textbf{(Last iterate rate)}.
    \begin{equation*}
        \left\|v - \frac{1}{k}(z^k - z^0)\right\| \leq \sqrt{\frac{1}{k}}\|z^0 - z_\varepsilon\| + \varepsilon\ ,
    \end{equation*}
    \item[] \textbf{(Difference rate)}.
    \begin{equation*}
        \min_{j \leq k}\|v - z^{j+1}-z^j\| \leq \sqrt{\frac{1}{k}}\|z^0 - z_\varepsilon\| + \varepsilon\ .
    \end{equation*}
\end{enumerate}
\end{theorem}
\begin{remark}
In the paper \cite{liu2019new}, this result is presented only for the difference of iterates, yet a simple modification of their argument proves the other two results.
\end{remark}
The theorem guarantees a rate of convergence that depends on a target accuracy $\varepsilon$. The rate could get worst as $\varepsilon \rightarrow 0$. Indeed, $z_\varepsilon$ could diverge as $\varepsilon$ goes to zero, see Example~\ref{ex:divergent} in Appendix~\ref{sec:counterexamples}. We will see in the next section that for LP it is possible to get rates that are independent of the accuracy.

\medskip

Since the algorithm of interest is PDHG, we might wonder whether or not its operator is firmly nonexpansive. It turns out that it is, but not with respect to the standard Euclidian norm.

\begin{proposition}
If $ \pssize \dssize \|A\|^2_2 < 1$, then the operator defined by a PDHG iteration is firmly nonexpansive with respect to the norm the $\|\cdot\|_M$ with $M$ defined as in \eqref{eq:cp-characterization}.
\end{proposition}
\begin{proof}
This is a known result \cite{he2014convergence}, yet we include a proof for the interested reader. A Schur complement argument proves that the condition $ \pssize \dssize \|A\|^2_2 < 1$ ensures that $M \succ 0$ is positive definite. Then a direct application of Proposition 4.2 in \cite{bauschke2008general} proves the result.
\end{proof}

\section{Sublinear convergence of nonexpansive operators}\label{sec:sublinear-rate}
This section presents the $O(1/k)$ convergence rate of the normalized iterates and the normalized average for nonexpansive operators. This rate applies to a broader class of operators than the previously known results (restated in Theorem \ref{wotao}) as it does not require the operator to be firmly nonexpansive. The resulting rate applies to many popular FOMs for convex optimization, including but not limited to PDHG \cite{chambolle2011first}, the Alternating Method of Multipliers (ADMM) or equivalently Douglas-Rachford Splitting (DRS) \cite{powell1978algorithms}, and Mirror-Prox \cite{nemirovski2004prox}.

Theorem \ref{thm:general-convergence} presents our main result in this section.

\begin{theorem}\label{thm:general-convergence}
Let $T$ be a nonexpansive operator for some norm $\|\cdot\|$ and define $v$ to be the minimum norm element in $\cl(\range(T-I))$. Then, for any $\varepsilon > 0$, there exists $z_\varepsilon$ such that the following two inequalities hold
\begin{enumerate}
    \item[] \textbf{(Average iterate rate).}
    \begin{equation}\label{eq:ave-not-closed}\left\|v - \frac{2}{(k+1)}\left(\bar z^k - z^0 \right)\right\| \leq \frac{4}{k+1}\|z^0-z_\varepsilon\|+ \varepsilon\ .\end{equation}
    \item[] \textbf{(Last iterate rate).}
    \begin{equation}\label{eq:iter-not-closed}\left\|v - \frac{1}{k}(z^k - z^0)\right\| \leq \frac{2}{k}\|z^0-z_\varepsilon\|+ \varepsilon\ .\end{equation}
\end{enumerate}
Furthermore, if $\range(T-I)$ is closed, then there exists a finite $z^\star$ such that $T(z^\star)=z^\star + v$ and for any such $z^\star$ and all $k$:
\begin{enumerate}
\setcounter{enumi}{2}
    \item[] \textbf{(Average iterate rate).}
    \begin{equation}\label{eq:ave-closed}\left\|v - \frac{2}{(k+1)}\left(\bar z^k - z^0 \right)\right\| \leq \frac{4}{k+1}\|z^0-z^\star\|\ .\end{equation}
    \item[] \textbf{(Last iterate rate).}
    \begin{equation}\label{eq:iter-closed}\left\|v - \frac{1}{k}(z^k - z^0)\right\| \leq \frac{2}{k}\|z^0-z^\star\|\ .\end{equation}
\end{enumerate}
\end{theorem}
\begin{remark}
We comment that when $\range(T-I)$ is not closed, Theorem \ref{thm:general-convergence} may not imply a $O(1/k)$ sublinear convergence rate. In fact, as $\varepsilon \rightarrow 0$, the vector $\|z_\varepsilon\|$ could grow to infinity, see Example~\ref{ex:divergent} in Appendix~\ref{sec:counterexamples} for a one dimensional example. When $\range(T-I)$ is closed, we obtain a $O(1/k)$ sublinear rate.
\end{remark}

\begin{remark} When $\range(T-I)$ is closed, the above result together with the lower bound proved later in Theorem~\ref{thm:identifiability} shows that the normalized iterates and normalized average of a nonexpansive operator exhibit a $\Theta\left(\frac{1}{k}\right)$ convergence rate. It is faster than the difference of iterates, by noticing that the difference of iterates converges at rate $\Theta\left(\frac{1}{\sqrt{k}}\right)$ (see Theorem 8 of \cite{davis2016convergence}).
\end{remark}

The next Lemma is used in the proof of Theorem \ref{thm:general-convergence}.
\begin{lemma} \label{lemma:char-v} Suppose the assumptions of Theorem \ref{thm:general-convergence}. Fix $\varepsilon > 0,$ then there exists a point $z_\varepsilon$, such that the following two inequalities hold for all $k \geq 0$:
\begin{enumerate}
    \item $\|T^{k+1}(z_\varepsilon) -T^k(z_\varepsilon) - v\| \leq \varepsilon$.
    \item $\|(T^k(z_\varepsilon) - z_\varepsilon) - kv\| \leq k \varepsilon.$
\end{enumerate}
Furthermore, if $\range (T-I)$ is closed, then there exists a point $z^\star$ such that $T(z^\star)=z^\star + v$.  For all such $z^\star$, for all $k \geq 0$:
\begin{enumerate}
\setcounter{enumi}{2}
    \item $T^{k+1}(z^\star) -T^k(z^\star) = v.$
    \item $T^k(z^\star) - z^\star = kv$.
\end{enumerate}
\end{lemma}

\begin{proof} Without loss of generality we assume that $\varepsilon \leq 1$. Notice $v \in \cl(\range(T-I))$, thus there exists $z_\varepsilon$ such that
\begin{equation}\label{eq:eps}
    \|T(z_\varepsilon) - z_\varepsilon - v\| \leq \frac{\varepsilon^2}{\max\{1, 2(\|v\|+1)\}} \ .
\end{equation}

We start by proving the first claim. Fix an arbitrary $k \geq 0$. We will make use of two facts in the proof. Since $T$ is a nonnexpansive operator, an application of the triangle inequality yields
\begin{align}\label{eq:diff-bound}
    \|T^k(z_\varepsilon) - T^{k-1}(z_\varepsilon) \| - \|v\| \leq \|T(z_\varepsilon) - z_\varepsilon \| - \|v\| \leq  \|T(z_\varepsilon) - z_\varepsilon - v\| \leq \frac{\varepsilon^2}{2(\|v\|+1)}\ .
\end{align}
Noticing $v$ is the nearest point to zero in $W = \cl(\range(T-I))$ with respect to the norm $\|\|$ and the set $W$ is convex, it follows from the optimality conditions of this problem that
\begin{equation}\label{eq:proj-v}
    \dotp{w}{v} \geq \|v\|^2 \qquad\text{for all }w \in W
    \ .
\end{equation}
Armed with these two facts, we derive for any arbitrary $k$:
\begin{align*}
    \|T^k(z_\varepsilon) - T^{k-1}(z_\varepsilon) - v\|^2 & = \|T^k(z_\varepsilon) - T^{k-1}(z_\varepsilon) \|^2-2\dotp{T^k(z_\varepsilon) - T^{k-1}(z_\varepsilon)}{v} +\|v\|^2 \\
    &\leq \|T^k(z_\varepsilon) - T^{k-1}(z_\varepsilon) \|^2-2\|v\|^2 +\|v\|^2\\
    &= \left(\|T^k(z_\varepsilon) - T^{k-1}(z_\varepsilon) \|+\|v\|\right)\left(\|T^k(z_\varepsilon) - T^{k-1}(z_\varepsilon) \|-\|v\|\right)\\
    &\leq \left(\|T(z_\varepsilon) - z_\varepsilon - v\|+2\|v\|\right)\frac{\varepsilon^2}{2(\|v\|+1)}\\
    &\leq \left(\varepsilon^2 +2\|v\|\right)\frac{\varepsilon^2}{2(\|v\|+1)} \leq \varepsilon^2\ ,
\end{align*}
where the first inequality utilizes \eqref{eq:proj-v} by noticing $T^k(z_\varepsilon) - T^{k-1}(z_\varepsilon) \in W$, the second inequality uses \eqref{eq:diff-bound} and the triangle inequality, the third inequality is from \eqref{eq:eps}, and the last inequality uses $\varepsilon\le 1$. This proves the first statement.

The second claim follows by induction. The base case $k = 0$ holds directly.  For the inductive step, assume that the statement holds for $k-1$. Then
\begin{align*}
    \|(T^k(z_\varepsilon) - z_\varepsilon) - kv\| \leq \|T^k(z_\varepsilon) -T^{k-1}(z_\varepsilon) - v\| + \| T^{k-1}(z_\varepsilon) - z_\varepsilon - (k-1)v\| \leq k \varepsilon\ ,
\end{align*}
where we used the first claim and the inductive hypothesis.

Furthermore, if $\range (T-I)$ is closed, the statements follow by taking $\varepsilon = 0$ in the previous proofs.
\end{proof}

\begin{proof}[Proof of Theorem \ref{thm:general-convergence}]
Let $z_\varepsilon$ be the point given by Lemma \ref{lemma:char-v}. We proceed to prove the first two statements.

1. It follows from $z^j=T^j(z^0)$ that
\begin{align*}
    \left\|\frac{2}{k(k+1)}\sum_{j= 1}^k (z^j  - z^0) - v\right\| &= \left\|\frac{2}{k(k+1)}\sum_{j= 1}^k \left(T^j(z^0) - z^0 - jv\right)\right\| \\
    &= \left\|\frac{2}{k(k+1)}\sum_{j= 1}^k \left((T^j(z^0) - T^j(z_\varepsilon)) + (z_\varepsilon - z^0) + (T^j(z_\varepsilon) - z_\varepsilon - jv) \right) \right\|\\
    &\leq \left\|\frac{2}{k(k+1)}\sum_{j= 1}^k \left(T^j(z^0) - T^j(z_\varepsilon)\right)\right\| + \frac{2}{(k+1)} \left\| z^0 - z_\varepsilon \right\|  + \varepsilon \ ,
\end{align*}
where the inequality uses Lemma~\ref{lemma:char-v} and the triangle inequality. Applying the triangle inequality to the first term yields
\begin{align*}
    \left\|\frac{2}{k(k+1)}\sum_{j= 1}^k \left(T^j(z^0) - T^j(z_\varepsilon)\right)\right\| &\leq \frac{2}{k(k+1)}\sum_{j= 1}^k \left\|T^j(z^0) - T^j(z_\varepsilon)\right\|\\ & \leq \frac{2}{k(k+1)}\sum_{j= 1}^k \left\|z^0 - z_\varepsilon\right\| = \frac{2}{(k+1)} \left\|z^0 - z_\varepsilon\right\| \ ,
\end{align*}
where the second inequality follows since $T$ is nonexpansive.

2. Notice that
\begin{align*}
    \left\|\frac{1}{k} (z^k -z^0) - v\right\| &= \left\|\frac{1}{k}\left((T^k(z^0) - z^0)  - (T^k(z_\varepsilon) - z_\varepsilon) + (T^k(z_\varepsilon) - z_\varepsilon - kv)\right)\right\| \\
    & \leq \left\|\frac{1}{k}\left(\left(T^k(z^0) - z^0\right)- (T^k(z_\varepsilon) - z_\varepsilon)\right)\right\| + \varepsilon\\
    & \leq \frac{1}{k} \left(\|T^k(z^0) - T^k(z_\varepsilon)\| + \|z^0 - z_\varepsilon\|\right) + \varepsilon\\
    & \leq \frac{2}{k}\|z^0 - z_\varepsilon\| + \varepsilon \ ,
\end{align*}
where the first inequality uses triangle inequality and Lemma~\ref{lemma:char-v}, and the last inequality is from non-expansiveness of $T$.

Finally, if $\range(T-I)$ is closed, we obtain the results following the same argument as above with $\varepsilon = 0$, using the results for closed $\range(T-I)$ from Lemma~\ref{lemma:char-v}.
\end{proof}

A drawback of \eqref{eq:ave-not-closed} and \eqref{eq:iter-not-closed} in Theorem \ref{thm:general-convergence}, as well as the results in Theorem~\ref{wotao}, is that the constants accompanying the rates depend on $\varepsilon$. Nonetheless, we can bypass this issue, using \eqref{eq:ave-closed} and \eqref{eq:iter-closed}, for problems where $\range (T-I)$ is closed. The next proposition guarantees that $\range (T-I)$ is indeed closed for a broad family of algorithms for solving LP.

\begin{proposition}\label{prop:poly}
Let $T : \RR^d \rightarrow \RR^d$ be an operator that can be decomposed as $T=T_k \circ \dots \circ T_1$ where $T_j$ is either an affine mapping or a projection onto a polyhedron. Then, $\range(T)$ is a finite union of polyhedra.
\end{proposition}
\begin{proof}
The proof follows inductively. Assume that $C = \range (T_{j} \circ \dots \circ T_{1})$ is a finite union of polyhedra. Without loss of generality, we can assume that $C$ is equal to a single polyhedron. Now we consider two cases:
\begin{enumerate}
    \item[] \textbf{Case 1}. Assume that $T_{j+1}$ is an affine transformation. This is a well-known consequence of Fourier-Motzkin elimination~\cite{Martin1999}.
    \item[] \textbf{Case 2}. Assume that $T_{j+1}$ is a projection onto a polyhedron $Q$. First, we start with an intuitive sketch of the proof and then formalize it. In this case, different pieces of the polyhedron $C$ are going to be projected to different faces of the polyhedron $Q$. Each one of these pieces is a polyhedron and since there are only finitely many faces of $Q$, the projection is a finite union of polyhedra.

    More formally, any polytope $Q$ defines a finite polyhedral partition of the space $\{P_F\}_{F \in \Delta}$ where $\Delta$ is the collection of faces of the polyhedron $Q$.\footnote{I.e., each cell $P_F \subseteq \RR^d$ is a polyhedron and $\cup_{F\in \Delta} P_F = \RR^d$.} Each cell $P_F$ corresponds to the region of the space that projects onto $F$, that is $\proj_Q(P_F) = F.$ Define a partition of the polyhedron $C$ as $\{C_F\}_{F \in \Delta}$ given by $C_F = P_F \cap C.$  Within each cell $P_F$ the projection $\proj_Q\mid_{P_F}$ is an affine transform. Thus, by Case 1 we have that $\proj_Q(C_F)$ is a polyhedron and thus $$T_{j+1} (C) = \proj_Q(C) = \bigcup_{F \in F} \proj_Q(C_F)$$
    is a finite union of polyhedra.
\end{enumerate}
\end{proof}

As a result of Proposition \ref{prop:poly} applied to the PDHG update for solving an LP problem, $\range(T-I)$ is a polyhedron, thus closed:
\begin{corollary}\label{cor:poly}
Let $T$ be the PDHG operator for an LP problem, then $\text{range}(T-I)$ is a finite union of closed polyhedra.
\end{corollary}
\begin{proof}
Notice that the operator $T-I$ is composite of linear operators and projection operators, thus we obtain the results by using Proposition \ref{prop:poly}.
\end{proof}

\begin{remark}
Proposition \ref{prop:poly} shows that $\text{range}(T-I)$ is closed when $T$ is an operator that corresponds to other first-order algorithms, such as ADMM and mirror-prox, to solve an LP problem. Further, one can extend the result to cover convex quadratic problems with polyhedral constraints.
\end{remark}

\section{The complete behavior of PDHG for solving LP problems}\label{sec:characterization}
In Figure \ref{fig:examples}, we saw low-dimensional examples of the dynamics of PDHG when solving LP problems in different feasibility settings. Indeed, such convergence/divergence dynamics generally hold when using PDHG to solve arbitrary LP problems. In this section, we present a complete description of the behavior of PDHG for feasible and infeasible LP problems and discuss how to recover the infeasibility certificate from the iterates of PDHG. The next theorem compiles the full characterization. The proof of this theorem will be deferred to Section \ref{sec:proof-char}.

\begin{theorem} \label{thm:characterization}
Consider the primal \eqref{primal} and dual \eqref{dual} problems. Assume that $\pssize \dssize \|A\|_2^2 < 1$, let $T$ be the operator induced by \eqref{eq:lpdhg}, and let $(z^k)$ be a sequence generated by the fixed-point iteration for an arbitrary starting point $z^0.$ Then, one of the following holds:
\begin{enumerate}
    \item If \textbf{both primal and dual are feasible}, then the iterates $(x^k, y^k)$ converge to a primal-dual solution $z^\star = (x^\star, y^\star)$ and $v = (T-I)(z^\star) = 0.$
    \item If \textbf{both primal and dual are infeasible}, then both primal and dual iterates diverge to infinity. Moreover, the primal and dual components of the infimal displacement vector $v = (v_x, v_y)$ give certificates of dual and primal infeasibility, respectively.
    \item If \textbf{the primal is infeasible and the dual is feasible},  then the dual iterates diverge to infinity, while the primal iterates converge to a vector $x^\star$. The dual-component $v_y$ is a certificate of primal infeasibility. Furthermore, there exists a vector $y^\star$ such that $v = (T-I)(x^\star, y^\star)$.
    \item If \textbf{the primal is feasible and the dual is infeasible}, then the same conclusions as in the previous item hold by swapping primal with dual.
\end{enumerate}
\end{theorem}
To show this characterization, we establish two intermediate results: first, the infimal displacement vector $v$ is nonzero if, and only if, either the primal or dual problems are infeasible; and second, the iterates $(x^k, y^k)$ ``converge'' to a well-defined ray of the form $(x^\star, y^\star) + \lambda v$ for $\lambda \in \RR_+$. The first result describes the asymptotic divergent behavior of the primal (resp. dual) iterates when the dual (resp. primal) problem is infeasible. The second one, ensures the asymptotic convergence of the primal (resp. dual) iterates without any normalization when the dual (resp. primal) problem is feasible. These two intermediate results are proved in Section \ref{sec:v-nnz} and Section \ref{sec:ray-convergence}, respectively.

\subsection{The infimal displacement vector recovers infeasibility certificates} \label{sec:v-nnz}

In Section \ref{sec:sublinear-rate}, we demonstrated that the differences of iterates, the normalized iterates, and the normalized average for a nonexpansive operator converge to the infimal displacement vector $v$.
Here, we show that the infimal displacement vector $v$ for PDHG applied to LP recovers infeasibility certificates whenever it is nonzero.  First, some simple properties of $v$.
\begin{lemma} \label{lem:infimal-displacement}
Consider the primal \eqref{primal} and dual \eqref{dual} problems. Assume that $\pssize \dssize \|A\|_2^2 < 1$, let $T$ be the operator induced by \eqref{eq:lpdhg}, and let $v = (v_x, v_y)$ be the infimal displacement vector of $T$.  Then $v_x \geq 0$, $A v_x = 0$, and $A^\top v_y \geq 0$.
\end{lemma}
\begin{proof}
From Theorem \ref{thm:pazy},
\begin{equation}\label{eq:klimit}
    \frac{1}{k}(x^k, y^k) \rightarrow v = (v_x, v_y) \text{ \ \  and \ \   }\frac{1}{k} (z_{k+1}-z_k) \rightarrow 0 \ .
\end{equation}

Notice that PDHG for LP has the following iteration update in terms of a differential inclusion,
\begin{equation}\label{cp:characterization}
     M \begin{bmatrix} x^k - x^{k+1}\\ y^k - y^{k+1}\end{bmatrix}  \in \begin{bmatrix}N_{\RR^n_+}(x^{k+1}) + A^\top y^{k+1} +  c \\ -Ax^{k+1} +  b\end{bmatrix}\ ,
\end{equation}
where this relation comes from \eqref{eq:cp-characterization} and \eqref{eq:calculus}. Dividing \eqref{cp:characterization} by $k$ and letting $k\rightarrow \infty$, we have from \eqref{eq:klimit} that
\begin{equation}\label{eq:Avy}
    0 \in \lim_{k\in \infty} N_{\RR^n_+}(x^k) +  A^\top \frac{1}{k} y_k \subseteq  -\RR^n_+ +  A^\top v_y \implies A^\top v_y \geq 0 \ ,
\end{equation}
where we utilize the fact that $N_{\RR^n_+}(x) \subseteq -\RR^n_+$ for any $x\in \RR^n_+$ and $\lim_{k\rightarrow \infty}\frac{1}{k} y_k = v_y$; and
\begin{equation}\label{eq:Avx}
    0=\lim_{k\in \infty} - \frac{1}{k} A x^k= -A v_x  \implies Av_x = 0 \ .
\end{equation}
Furthermore, note that $v_x \geq 0$ since $v_x = \lim_{k\rightarrow \infty} x^k/k$ and $x^k \geq 0$ for all $k$.
\end{proof}
Now we derive the main result of this section.
\begin{proposition} \label{thm:infea-iff-nonzerov}
Consider the primal \eqref{primal} and dual \eqref{dual} problems. Assume that $\pssize \dssize \|A\|_2^2 < 1$, let $T$ be the operator induced by \eqref{eq:lpdhg}, and let $(z^k)_{k\in\NN}$ be a sequence generated by the fixed-point iteration for an arbitrary starting point $z^0.$ Then, the primal problem \eqref{primal} is infeasible if and only if $v_y$ is a nonzero vector, and in this case, $v_y$ is an infeasibility certificate for the primal problem. Analogously, the dual problem \eqref{dual} is infeasible if and only if $v_x$ is a nonzero vector, and in this case, $v_x$ is an infeasibility certificate for the dual problem.
\end{proposition}

\begin{proof}
To establish the first implication in this result we have to prove that if $v_y$ is non-zero, then $v_y$  is an infeasibility certificate for the primal problem, namely, $$A^\top v_y \geq 0 \text{\  and\  } b^\top v_y < 0 \ ,$$ thus the primal problem is infeasible. Similarly if $v_x$ is non-zero, then $v_x$  is an infeasibility certificate for the dual problem, namely
$$Av_x = 0, v_x \geq 0,\text{\  and\  }c^\top v_x < 0 \ ,$$
thus the dual problem is infeasible. We proved all the nonstrict inequalities in Lemma \ref{lem:infimal-displacement}, so it suffices to show the strict ones.

First, consider the case when $v_x \neq 0$. Let $B = \{i \in [n] \mid (v_x)_i> 0\}$ and let $N = \{i \in [n] \mid (v_x)_i = 0\},$ then $B \cup N = [n]$ by noticing $v_x\ge 0$ (from Lemma~\ref{lem:infimal-displacement}).  Given a vector $x$, let $x_B$ be the vector of entries of $x$ with indices in $B$; similarly given a matrix $A$, let $A_B$ be the submatrix with columns of $A$ with indices in $B$. Then for any $i \in B$, we have $(v_x)_i>0$, thus there exists some $K$ such that $(x^k/k)_i > 0$ for all $k \geq K$, and furthermore
$$(x^{k+1})_B = (x^k)_B - \pssize  A_B^\top y^k -  \pssize c_B\ .$$
Taking the limit $k\rightarrow \infty$ and noticing $\lim_{k\rightarrow\infty}(x^{k+1})_B - (x^k)_B=(v_x)_B$, we obtain
$$(v_x)_B = \lim_{k\rightarrow \infty} -\pssize (A_B^\top y^k + c_B)\ .$$
Thus it holds that
\begin{equation}
    c^\top v_x = c_B^\top (v_x)_B = -\frac{1}{\pssize}\|v_x\|^2_2 -  \lim_{k\rightarrow \infty} (y^k)^\top A_B (v_x)_B = - \frac{1}{\pssize}\|v_x\|^2_2 < 0\ ,
\end{equation}
where the last equality uses $A_B (v_x)_B = Av_x = 0.$ Combining with $v_x\ge 0$ and $Av_x=0$ proves that $v_x$ is a certificate of infeasibility whenever it is nonzero.

Second, consider the case when $v_y \neq 0$. By taking $k\rightarrow \infty$, we have
\begin{equation}\label{eq:v_y}
    v_y = \lim_{k\rightarrow \infty} y_{k+1}-y_k = \lim_{k\rightarrow \infty} \dssize A (2x^{k+1}-x^k) - \dssize b = \lim_{k\rightarrow \infty} \dssize A x^{k+2} - \dssize b \ ,
\end{equation}
where the second equality uses the update rule \eqref{eq:lpdhg}, and the third equality uses $\lim_{k\rightarrow \infty} 2x^{k+1}-x^k = \lim_{k\rightarrow \infty} x^{k+1}+v_x = \lim_{k\rightarrow \infty} x^{k+2}$.

Now we claim the following two facts:
\begin{fact}\label{lem:property}
There exists some $K$ such that if $(A^\top v_y)_i > 0$ then $x^k_i = 0$ for all $k\ge K$.
\end{fact}
\begin{fact}\label{lem:second-property}
The support (nonzero components) of $A^\top v_y$ satisfies $\supp(A^\top v_y) \subseteq N.$
\end{fact}
The first fact is because if $(A^\top v_y)_i > 0$ then we have that $(A^\top y^k/k)_i \geq (A^\top v_y)_i/2 > 0$ for large enough $k.$ Dividing \eqref{cp:characterization} by $k$ yields
$$-\frac{1}{k} (A^\top y^k)_i + \frac{1}{\pssize k}\left(x^k-x^{k+1} -\pssize c\right)_i \in N_{\RR^n_+}(x^{k+1}_i) \ .$$
For large enough $k$, the second term on the left-hand side of the inclusion will be as small as $(A^\top y^k/k)_i/2$ and hence the sign of entire expression on the left-hand side will be negative. If $N_{\RR_+}((x^{k+1})_i)$ contains a negative number, then $(x^{k+1})_i = 0$, which implies that $(x^{k+1})_i = 0$ for large enough $k$.

The second fact is because for any entry $i$ in the support of $A^\top v_y$, namely $(A^\top v_y)_i>0$, it follows from the first part that $(x^k)_i=0$ for all $k$ large enough, thus $(v_x)_i= \lim_{k\rightarrow \infty} \frac{1}{k} x^k_i=0$, which proves the second fact by the definition of the set $N$.

Returning to the proof of Proposition \ref{thm:infea-iff-nonzerov}, notice that
\begin{equation}\label{eq:cup}
    \lim_{k\rightarrow \infty} v_y^\top A x^{k}= \lim_{k\rightarrow \infty} \sum_{i\in N} (A^\top v_y)_i x^k_i =0 \ ,
\end{equation}
where the first equality uses Fact~\ref{lem:second-property}, and the second equality uses Fact~\ref{lem:property}. Therefore, it holds that
\begin{equation*}
    v_y^\top b = -\frac{1}{\dssize} \|v_y\|^2_2 + \lim_{k\rightarrow\infty} v_y^\top A x^{k+2}= -\frac{1}{\dssize} \|v_y\|^2 < 0 \ ,
\end{equation*}
where the first inequality uses \eqref{eq:v_y} and the second equality is from \eqref{eq:cup}. Together with \eqref{eq:Avy}, we know $v_y$ is an infeasibility certificate for the primal problem.

Now we turn to the inverse direction. Recall that it follows from the closedness of the set $\range(T-I)$ that there exists a pair $z^\star=(x^\star, y^\star)$ such that $T(z^\star)=z^\star+v$. If the dual problem is infeasible, we will show that $v_x\not=0$ by contradiction. Assume $v_x=0$; then it follows from the update rule \eqref{eq:lpdhg} that
$$
x^\star=\proj_{\RR_+^n}(x^\star- \pssize (A^\top y^\star + \pssize c)) \ ,
$$
thus $A^\top y^\star + \pssize c\ge 0$ by noticing $x^\star\ge 0$, which contradicts the assumption that the dual problem is infeasible. If the primal problem is infeasible, we will show that $v_y\not=0$ by contradiction. Suppose $v_y=0$, then it follows from the update rule \eqref{eq:lpdhg} that
$$
y^\star=y^\star+\dssize A(2(x^\star+v_x)-x^\star)-\dssize b \ ,
$$
thus $Ax^\star=b$ by noticing $A v_x=0$ from \eqref{eq:Avx}. Furthermore, we know $x^\star\ge 0$, thus the primal is a feasible problem, which contradicts with assumption.

This concludes the proof.

\end{proof}

\subsection{The iterates converge to a ray}\label{sec:ray-convergence}
Combining facts from the previous sections we know that if both primal and dual problems are feasible then the iterates (without normalization) will converge to a solution, and when both primal and dual problems are infeasible then the normalized iterates converge to a vector $(v_x, v_y)$ with nonzeros on both primal and dual components. Yet the techniques used to prove these results do not explain what happens when one of the problems is feasible and the other one is infeasible. In this scenario the convergence of the primal and dual iterates happen at different scales, one with normalization by $\frac{1}{k}$ and the other without it. In this section, we fill in this gap by showing that the iterates of PDHG always converge to ray with direction $v$, emanating from a point $z^\star$. In turn, this allows us to connect the convergence results for the two scales.

\begin{definition}[\textbf{Ray}] Given a starting point $z^\star \in \RR^{n+m}$ and a direction $v\in\RR^{n+m}$, we define their \emph{ray} as
$$[z^\star, v] = \{z^\star + \lambda v \mid \lambda \in \RR_+ \}\ .$$

\end{definition}
With this definition at hand we can now state the main result of this section.
\begin{theorem} \label{thm:fea-infea} %
Consider the primal \eqref{primal} and dual \eqref{dual} problems. Assume that $\pssize \dssize \|A\|^2_2 < 1$, let $T$ be the operator induced by \eqref{eq:lpdhg}, and let $(z^k)_{k\in\NN}$ be a sequence generated by the fixed-point iteration for an arbitrary starting point $z^0.$ Then, the iterates of PDHG converge to a ray $[z^\star, v]$, in particular $$\|z^k-z^\star- k v\| \rightarrow 0  \quad \text{for some} \quad z^\star \in (T-I)^{-1}(v)\ .$$
\end{theorem}

To prove this result, we establish a connection between the iterates of PDHG applied to the original (possibly infeasible) problem and the iterates of PDHG applied to a feasible auxiliary LP problem.
Let us start by defining this auxiliary problem. Define the index sets
\begin{align}\begin{split}\label{eq:indices}
B &= \{i \in [n] \mid (v_x)_i > 0\}, \\
N_1 &= \{i \in [n] \mid (v_x)_i = 0, (A^{\top} v_y)_i=0\}, \\
N_2 &= \{i \in [n] \mid (v_x)_i = 0, (A^{\top} v_y)_i>0\}\ .
\end{split}
\end{align}
Define the operator $\tilde{T}: \RR^{n +m} \rightarrow \RR^{n+m}$ given by $\tilde{T} (z) := z^+$ with
\begin{align}\label{eq:T1update}
\begin{split}
        (x^+)_B &= x_B - \pssize A^\top_B y - \pssize c_B -(v_x)_B \\
    (x^+)_{N_1} &= \proj_{\RR_+^{|N_1|}}(x_{N_1}- \pssize A^\top_{N_1} y - \pssize c_{N_1}) - (v_x)_{N_1} \\
    (x^+)_{N_2} &= - (v_x)_{N_2}\\
     y^+ &= y + \dssize A(2x^+ -x) - \dssize b -v_y\ .
\end{split}
\end{align}
In turn, this is a PDHG operator for the auxiliary LP problem:
\begin{align}\begin{split}\label{eq:aux-prob}
    \text{minimize}  &\quad (c_B + (v_x)_B/\pssize)^\top x_B + c_{N_1}^\top x_{N_1} + c_{N_2}^\top x_{N_2}\\
    \text{subject to} &\quad A_B x_B + A_{N_1}x_{N_1} + A_{N_2} x_{N_2}  = b + \frac{v_y}{\dssize}\\
    &\quad x_{N_1} \geq 0, \ x_{N_2} = 0\ .
    \end{split}
\end{align}
Then we claim the following connection between $T$ and $\tilde T$.

\begin{proposition}\label{prop:claim} Given an arbitrary initial solution $z^0,$ there exists a large enough $K \in \NN$ such that
\begin{equation}\label{eq:claim}
    \tilde{T}^k(z^K)=T^k(z^K)-kv \qquad \text{for all }k \geq 0\ .
\end{equation}
\end{proposition}
\begin{proof}
For any initial solution, we know that there exists some $K$ such that it holds for any $k\ge K$ that \begin{equation} \label{eq:fixed-supp}(x^k)_B >0 \text{ and } (x^k)_{N_2}=0.\end{equation} The former is because $(v_x)_B>0$, and the latter follows from Fact~\ref{lem:property}. With some abuse of notation, we let $z^0\leftarrow z^K$, so that we may study the iterates starting at $z^0$ (rather than starting at $z^K$). From Lemma~\ref{lem:infimal-displacement}:
\begin{equation*}
    v_x \ge 0, \quad A v_x =0, \quad \text{and } \quad A^\top v_y \ge 0\ .
\end{equation*}
In addition, from the converse of Fact~\ref{lem:property},
\begin{equation*}
(A^{\top} v_y)_B=0\ .
\end{equation*}

We show the stated claim by induction. Denote $z^k=T^k(z^0)$ and $\tz^k=\tilde{T}^k(z^0)$. First, \eqref{eq:claim} holds with $k=0$.  Now suppose \eqref{eq:claim} holds for $k$, and consider $k+1$. Then by induction we have $\tz^{k+1}= \tilde T(z^k-kv)$, thus it holds by \eqref{eq:T1update} that
\begin{align*}
    (\tx^{k+1})_B&= (x^{k})_B - k(v_x)_B - \pssize A^\top_B (y^k-kv_y) - \pssize c_B -(v_x)_B \\
    &= (x^{k+1})_B -(k+1) (v_x)_B
\end{align*}
where the second equality utilizes the update rule of PDHG by noticing $A^\top_B v_y=0$ and $(x^{k+1})_B>0$.
For the components in $N_1$ we get
\begin{align*}
    (\tx^{k+1})_{N_1}&= \proj_{\RR_+^{|N_1|}}((x_{k})_{N_1} - k(v_x)_{N_1} - \pssize A^\top_{N_1} (y^k-kv_y) - \pssize c_{N_1}) -(v_x)_{N_1} \\
    &= \proj_{\RR_+^{|N_1|}}((x_{k})_{N_1}  - \pssize A^\top_{N_1} y^k - \pssize c_{N_1}) \\
    &= (x^{k+1})_{N_1} -(k+1) (v_x)_{N_1} \ ,
\end{align*}
where the second equality follows from $A^\top_{N_1} v_y=0$ and $(v_x)_{N_1}=0$, the third one utilizes $(v_x)_{N_1}=0$ and the update rule of PDHG. Similarly, for the $N_2$ block
\begin{align*}
    (\tx^{k+1})_{N_2}&= -(v_x)_{N_2}=0=(x^{k+1})_{N_2}-(k+1)(v_x)_{N_2} \ ,
\end{align*}
where the equations follow from $(x^{k+1})_{N_2}=0$ and $(v_x)_{N_2}=0$.
Finally, for the dual iterates we have
\begin{align*}
    \ty^{k+1}&=y^k-kv_y + \dssize A(2 \tx^{k+1} - \tx_{k}) - \dssize b - v_y \\
    &=y^k + \dssize A(2 (x^{k+1}-(k+1)v_x) - (x^k-kv_x)) - \dssize b - (k+1)v_y \\
    &=y^k + \dssize A(2 x^{k+1} - x^k) - \dssize b - (k+1)v_y \\
    &=y^{k+1} - (k+1)v_y
\end{align*}
where the third equality utilizes $Av_x=0$, and the last equality is from the update rule of PDHG.
 \end{proof}
Equipped with this proposition we can now prove the theorem.
\begin{proof} [Proof of Theorem~\ref{thm:fea-infea}]
Since $\tilde T$ is a PDHG operator, it is firmly nonexpansive with respect to $\|\cdot\|_M$. Thus, if $\tilde T$ has a fixed point, then the iteration $\tilde T^k(z^K)$ should converge to it. To see that $\tilde T$ has a fixed point, let $z^\star$ be a point such that $(T-I)(z^\star) = v$ and let $K$ be the iteration after which $\tilde T^k(T^K(z^\star)) = T^{k+K}(z^\star) - kv$, which exists thanks to Proposition \ref{prop:claim}. We claim that $T^K(z^\star)$ is a fixed point of $\tilde T$. To see this, note that
$$\tilde T(T^K(z^\star)) = T^{K+1}(z^\star) - v = T^K(z^\star),$$
where the last equality follows from Lemma \ref{lemma:char-v}.

Now, let $z^0$ an arbitrary initial point and $K$ be the in \eqref{eq:claim}. Now that we know that the set of fixed points of $\tilde T$ is nonempty, we can define $z^\star = \lim_{k\rightarrow \infty} \tilde T^k(T^K(z_0))$. We will prove that $z^\star$ satisfies $(T-I)(z^\star) = v$. Due to Proposition \ref{prop:claim}, we know that
$$z^\star = \tilde T(z^\star) = T(z^\star) - v.$$

Finally, using decomposition \eqref{eq:claim} we get
\begin{equation}\label{eq:ray-conv}
\|z^{k+K} - z^\star - k v\| = \|\tilde T(z^k) - z^\star\| \rightarrow 0\ .\end{equation}
The statement of theorem claimed this convergence where the coefficient accompanying $v$ is $(k+K)$. We can get around this by setting $z^\star \leftarrow z^\star - K v$, a point that also satisfies $(T-I)(z^\star) =v$ thanks to Lemma \ref{lemma:char-v}. This establishes the result.
\end{proof}

\subsection{Proof of Theorem~\ref{thm:characterization}} \label{sec:proof-char}
\begin{proof}
As a direct result of Proposition \ref{thm:infea-iff-nonzerov}, we know that if both primal and dual are feasible, then $v=0$ and Theorem~\ref{thm:fea-infea} ensures that PDHG converges to an optimal solution (or equivalently a fixed point of $T$). If primal (and/or dual) is infeasible, then the dual iterate of PDHG (and/or primal) diverges to infinity, and the diverging direction recovers primal (and/or dual) infeasibility certificate.

Thus, the only thing left to prove is the final conclusion of item 3 (and analogously item 4). Assume that the primal problem is infeasible and the dual is feasible. By Proposition~\ref{thm:infea-iff-nonzerov} we know that $v_x = 0$. Then, Theorem~\ref{thm:fea-infea} guarantees the existence of some $z^\star = (x^\star, y^\star)$ such that $x^k \rightarrow x^\star + kv_x = x^\star$ and $(T-I)(z^\star) =v$. The proof for the case where the primal is feasible and the dual is infeasible follows from an analogous argument. This completes the proof of the theorem.\end{proof}

\section{Finite time identifiability and eventual linear convergence}
\label{sec:identifiability}

In this section, we introduce a nondegeneracy condition that ensures that after a finite amount the difference of iterates converges linearly to the infimal convergence vector. To show this, we demostrate under said condition the iterates ``identify'' the support of $x^k$, i.e., the support freezes after a finite number iterations. Finite-time identification has a long history in the analysis of iterative algorithms for feasible problems \cite{Dunn87, Calamai-More87, Burke-More88, Burke90, Flam92, lewis_active, liang2018local}. Roughly speaking, these algorithms' behaviors exhibit two phases: a first one that only takes finitely many steps but suffers from slow sublinear convergence, and then a second one after the active set is identified where the convergence is significantly faster and becomes linear.

For PDHG, finite-time identifiability is known to hold for \emph{feasible} minimax problems under suitable nondegeneracy conditions \cite{liang2018local}. In contrast, here we study this phenomenon for infeasible LP problems. We demonstrate that \emph{even when there is no primal (and/or dual) feasible solution, active set of the iterates with respect to an auxiliary feasible LP problem is fixed after finitely many iterations}.

Recall that the iterates of PDHG converge to a ray $[z^\star, v] = \{z^\star + k v \mid k \in \NN\}$ where $z^\star$ is a solution to the feasible LP problem given by \eqref{eq:aux-prob}. Consider the constraint set defined by said auxiliary problem, that is
\begin{equation} \label{eq:aux-p-system}
    Ax = b + \frac{v_y}{\dssize}, \qquad x_{N_1} \geq 0 \qquad \text{and} \qquad x_{N_2} = 0\ .
\end{equation}
Here, the active set is the set of inequality constraints that attain their extreme values, namely, $\{i \in N_1 \mid x^\star_i = 0\}$. Note that when the problem is feasible, $N_1 = \{1, \dots, n\}$ and thus the constrained set defined by the auxliary problem \eqref{eq:aux-p-system} matches that of the original problem.

Now, we introduce the nondegeneracy condition for possibly infeasible problems, which generalizes the classical identifibility theory of PDHG for feasible LP problems \cite{lewis_active, liang2018local}. Similar variations of the nondegeneracy condition have appeared in numerous works that deal with finite time identifiability. Further generalizations of this idea have led to conditions beyond the context of optimization, we refer the interested reader to \cite{lewis2018partial} for a perspective from differential geometry.
\begin{definition}
A ray $[z^\star, v]$ is \emph{nondegenerate} if for any $i \in N_1$ (recall $N_1$ is defined in \eqref{eq:indices}), the pair $(x^\star_i, (A^\top y^\star +c)_i)$ satisfies strict complementarity with respect to the auxiliary problem \eqref{eq:aux-prob}: $x_i^\star > 0$ if, and only if, $(A^\top y^\star + c)_i = 0$.
\end{definition}
Although here we chose to define nondegeneracy in terms of the extreme point $z^\star$, this definition is independent of the point we take in the ray $[z^\star, v].$ This follows easily from the fact that $(v_x)_{N_1}$ and $(A^\top v_y)_{N_1}$ are zero vectors. Additionally, notice that when the original problem is feasible, this definition reduces to the classical strict complementarity of the original problem.

The next lemma presents the finite time identifibility of PDHG for infeasible LP:
\begin{lemma}
Suppose $[z^\star, v]$ is a nondegenerate ray. Then, every PDHG iterate sequence $z^k = (x^k, y^k)$, converging to the ray $[z^\star, v]$, fixes the active set of \eqref{eq:aux-p-system} after finitely many steps. Furthermore, this ensures that the support of $x^k$ is fixed for all large enough $k$.
\end{lemma}

\begin{proof}
First let us prove that the active set of \eqref{eq:aux-p-system} is identified in finite time. Notice that this is equivalent to saying that the support of $x_{N_1}^k$ freezes after finitely many iterations. Let $i\in N_1$, due to strict complementary it is enough to consider two cases:
\begin{enumerate}
    \item[] \textit{Case 1}. Assume that $(Ay^\star + c)_i > 0$, then complementary slackness implies $(x^k)_i \rightarrow 0$. By construction, we should have $\pssize \cdot (A^\top y^k + c)_i > x^k_i$ for all $k$ large enough. After this condition starts to hold, the PDHG update at $i$ gives
    $$x^{k+1}_i = \left(x^k_i - \pssize\cdot (A^\top y^k + c)_i\right)_+ = 0\ .$$
    Hence, we have $x^k_i = 0$ for all large $k$.
    \item[] \textit{Case 2}. Assume that $x^\star_i > 0$, then after finitely many iterations we have $x^k_i > 0$.
\end{enumerate}
Thus, the support of $x_{N_1}^k$ is identified in finite time. Now, we argue that the same happens to the support of $x^k$. Assume that $i \in B$, then $(v_x)_i > 0$ and consequently for all large $k$ we have $x^k_i/k > 0$, as we wanted. Lastly, if $i \in N_2$ then $A^\top v_y > 0$ and Fact~\ref{lem:property} guarantees that $x^k_i = 0$ for large enough $k$. This concludes the proof.
\end{proof}

When nondegeneracy holds, PDHG eventually identifies the support of the primal iterate $x^k$. This simplifies the form of each iteration. Let $S$ be the support of any $x^k$ with $k$ large enough. The projection to the positive orthant applied by PDHG \eqref{eq:lpdhg} becomes a projection to the subspace $\{x \mid \supp(x) = S\}$. As a consequence, one can recast each iteration \eqref{eq:lpdhg} as an affine transformation:
\begin{equation}\label{eq:affine-update}
    \begin{bmatrix} x^{k+1}\\y^{k+1}\end{bmatrix} = \underbrace{\begin{bmatrix}I & - \pssize D A^\top \\ \dssize A D & I - 2\dssize \pssize AD A^\top\end{bmatrix}}_{Q := } \begin{bmatrix}x^k \\ y^k\end{bmatrix} - \underbrace{\begin{bmatrix}\pssize D c \\2 \dssize \pssize AD c +\dssize b\end{bmatrix}}_{p:=}
\end{equation}
where $D$ is a diagonal matrix with ones on the indices $(i, i)$ such that $i \in S$
and zeros everywhere else, and matrix $Q$ and vector $p$ are  defined in \eqref{eq:affine-update}.
\begin{theorem}[\textbf{Eventual linear convergence under nondegeneracy}]\label{thm:identifiability} Consider the primal \eqref{primal} and dual \eqref{dual} problems. Assume that $\pssize \dssize \|A\|_2^2 < 1$, and let $(z^k)_{k\in\NN}$ be a sequence generated by PDHG. Suppose that the iterates $z^k = (x^k, y^k)$ converge to a nondegenerate ray.
Then, the $k$th power of $Q$ converges $\lim_{k\rightarrow \infty} Q^k = Q^\infty$ to a projection matrix, and there exists a finite $K$ such that for any $k \geq 0$, the active set of $x^{K+k}$ is fixed. Furthermore, there exist positive constants $\underline{\mu}, c_1, c_2, c_3, c_4 > 0$ such that the following holds:
\begin{enumerate}
    \item \textbf{Linear convergence of the differences}. For any $\mu \in (\sqrt{1-\pssize\dssize\sigma_{\min}(A)^2} , 1)$, the differences $z^{K+k+1}-z^{K+k}$ converge at a linear rate to the infimal displacement vector $v$, i.e.,
    \begin{equation}
       \underline{\mu}^k\| (Q-I)z^K+(Q^\infty-I)p\|_2 \leq  \| z^{K+k+1} - z^{K+k} - v\|_2 \leq \mu^{k}\|(Q-I)z^K -p\|_2
    \end{equation}
    for all sufficiently large $k$.
    \item \textbf{Sublinear convergence of the iterates}. The normalized iterates converge to $v$ at a $\Theta \left(\frac{1}{k}\right)$ rate, i.e.,
    \begin{equation}\label{eq:item2}
       \frac{c_1}{k} \|(I-Q^\infty)p\|_2 \leq \left\| \frac{1}{k} {z}^{K+k} - v \right\|_2 \leq \frac{c_2}{k}\left(\|(Q-I)^\dagger(I-Q^\infty)p - z^K\|_2 + \|z^K\|_2\right)
    \end{equation}
    for all sufficiently large $k$.
    \item \textbf{Sublinear convergence of the average}.  The normalized average converges to $v$ at a $\Theta(\frac{1}{k})$ rate, i.e.,
    \begin{equation}\label{eq:item3}
        \frac{c_3}{k+1} \|(I-Q)p\|_2 \leq \left\| \frac{2}{k(k+1)} \sum_{j=1}^k z^{K+j} - v \right\|_2  \leq \frac{c_4}{k+1}\left(\|(Q-I)^\dagger (I-Q^\infty)p - z^K\|_2\ + \|z^K\|_2\right)
    \end{equation}
    for all sufficiently large $k$.
\end{enumerate}
\end{theorem}
A few remarks are in order.
\begin{remark}
Although equations \eqref{eq:item2} and \eqref{eq:item3} state a bound for the normalized iterates and normalized average of PDHG started from $z^K$, this result implies the same asymptotic bounds (with sightly worse constants) for the normalized iterates and normalized averaged started from $z^0$. Thus, the result concludes that under nondegeneracy, the difference of iterates converges much faster than the iterates and average.
\end{remark}
\begin{remark}
The proof of this result shows the same rates for Bilinear games. Recall that a Bilinear game is a minimax problem of the form
$$ \min_{x\in \RR^n}\max_{x\in \RR^m} c^\top x + \dotp{Ax}{y} - b^\top y\ .$$
For these problems, the updates of PDHG (Algorithm \ref{alg:pdhg}) take the form of \eqref{eq:affine-update} with $D = I$. Thus, all the arguments in the proof of this result follow with $K = 0$. In particular, this shows that the upper bound derived in Theorem \ref{thm:general-convergence} is tight for Bilinear minimax problems.
\end{remark}
\begin{remark}
Every LP problem that has an optimal solution furthermore has at least one primal-dual solution that satisfies strict complementarity~\cite[Theorem 2.35]{Martin1999}. Consequently, every problem has at least one nondegenerate ray. Thus there exists at least one initial point $z^0$, such that if PDHG is initialized at this point, then the iterates converge to the nondegenerate ray and thus enjoy linear convergence.
\end{remark}
\begin{proof} [Proof of Theorem \ref{thm:identifiability}]
We start by making a few simplifying assumptions. First we assume $D = I$. If that is not the case, we can consider a submatrix of $A$ where we trim out the columns indexed by $\{i \in n \mid D_{ii} = 0\}$. This has no effect in the end result since for all these indices $x^k_i = 0$, and thus it does not affect the nonzero entries of $x^k$ nor the entries of $y^k$.
Furthermore, we assume that $A$ is diagonal, otherwise we can decompose the matrix $Q$ as
\begin{equation}
    Q = \begin{bmatrix}V & 0 \\ 0 & U\end{bmatrix} \begin{bmatrix} I & \pssize \Sigma^\top \\ \dssize \Sigma & I-2\pssize \dssize \Sigma \Sigma^\top \end{bmatrix}\begin{bmatrix}V^\top & 0 \\ 0 & U^\top \end{bmatrix}
\end{equation}
where $A = U\Sigma V^\top$ is the SVD decomposition of $A$.  Changing basis is equivalent to applying an orthogonal map and thus does not alter the metric nor the algorithm (which reduces to an affine transform as soon as the active set is identified), thus we can change the basis to enforce this assumption.
Notice that the columns and rows of $Q$ can be further permuted so that it becomes a block diagonal matrix of the form
\begin{equation}\label{eq:simplified-Q}
    Q = \begin{bmatrix} B_1 & & & \\ & \ddots & & \\ &&B_{\min\{n, m\}} & \\ &&&I\end{bmatrix} \qquad \text{where}\qquad B_i = \begin{bmatrix}
    1 & -\pssize \sigma_i \\
    \dssize \sigma_i & 1-2 \dssize \pssize \sigma_i^2
    \end{bmatrix} \ ,
\end{equation}
where $\sigma_i$ is a the $i$th singular value of $A$. Note that when $n > m$ (or $n < m$), we might also have block corresponding to an identity of size $n-m$ (resp. $m-n$), the arguments below can be easily extended to cover the identity block (as it follows from the rationale applied when $\sigma_i = 0$) and so we assume that $n = m$. Thus, from now on $Q$ has the form \eqref{eq:simplified-Q} with $n = m$ and hence without the last identity block.

Now, we can compute closed-form formulas for the three certificate candidates \eqref{eq:diff_iterates}-\eqref{eq:normalized_average}. Let $K$ be the smallest integer after which the PDHG iteration can be written as \eqref{eq:affine-update}. By recursively expanding, we obtain
\begin{align}\begin{split}\label{eq:current-iterate}
        z^{k+1 + K} &= Qz^{k + K} - p \\
        &= Q^2z^{k-1 + K} - Qp - p\\
        & \quad \vdots \\
        & = Q^{k+1} z^{K} - \sum_{i=0}^{k}Q^i p\ .
\end{split}
\end{align}
If we take the difference between to consecutive iterates, this yields
\begin{align}\label{eq:diff}
       z^{k+1+K} - z^{k+K}  &= Q^k(Q-I)z^{K} - Q^k p = Q^k((Q-I)z^K - p)\ .
\end{align}
On the other hand, summing the first $k$ iterates \eqref{eq:current-iterate} gives
\begin{equation}
    \sum_{j=1}^{k} z^{j+K} = \sum_{j={1}}^{k} Q^j z^{K} - \sum_{j=1}^{k}\sum_{l=0}^{j-1}Q^l p\ .
\end{equation}
We will show that $Q^k$ converges to a matrix $Q^\infty$. We define the matrix $Q^\infty$ as a block diagonal matrix
\begin{equation}\label{eq:simplified-Q-infty}
    Q^\infty = \begin{bmatrix} B_1^\infty & & \\ & \ddots &  \\ &&B_n^\infty \end{bmatrix} \qquad \text{where}\qquad B_i^\infty = \begin{cases}I_2 & \text{if }\sigma_i = 0 \\ 0 & \text{otherwise.}\end{cases}
\end{equation}
Since each block is independent of each other, we can analyze $Q^k$ by studying $B_i^k$. A simple calculation reveals that the $i$th block has two eigenvalues of the form
\[\lambda_i^{\pm} = (1-\pssize \dssize \sigma_i^2) \pm i\left(\pssize \dssize \sigma_i^2 \left(1-\pssize\dssize \sigma_i^2\right)\right)^\frac{1}{2}\ .\]
Taking the norm, we find $\rho(B_i) = |\lambda_i^{\pm}| = \sqrt{1-\pssize\dssize \sigma^2_i}.$ Then, we have that the iterated product of the $i$th block $B_i^k$ converges to $B_i^\infty.$ To see this, consider two cases:
\begin{enumerate}
    \item[] \emph{Case 1}. Assume that $\sigma_i > 0$. By assumption $0 < 1-\pssize\dssize \sigma^2_i < 1$ hence $B^k_i \rightarrow 0 = B^\infty_i.$ This follows since the spectral radius $\rho(B_i^k) = (1-\pssize\dssize \sigma^2_i)^{k/2} \rightarrow 0$ as $k$.
    \item[] \emph{Case 2}. Assume that $\sigma_i = 0$. Then $B_i^k = B_i = I = B_i^\infty$.
\end{enumerate}
The matrix $Q^\infty$ turns out to be the projection onto the kernel of $Q-I$ (that is, $Q^\infty (Q-I) = 0$). We use $\Delta \subseteq [n]$ to denote the set of indices such that $\sigma_i > 0.$

\textbf{Differences}. We start by analyzing the differences \eqref{eq:diff}.
To prove the upper bound, we expand
\begin{align*}
\|z^{K+k+1} -z^{K+k} - v\|_2 &= \|Q^k(((Q-I)z^K-p)-Q^\infty ((Q-I)z^K-p)\|_2\\
&\leq \|Q^k-Q^\infty\|_2\|(Q-I)z^K-p\|_2\\
& = \max_{i \in \Delta}\|B_i^k\|_2\|(Q-I)z^K-p\|_2\\
&\leq \mu^{k}\|(Q-I)z^K-p\|_2\ ,
\end{align*}
where the first equality comes from taking $\lim_{k \rightarrow \infty}$ of (\ref{eq:diff}), the first inequality used the fact that $\|Wz\|_2 \leq \|W\|_2 \|z\|_2$ for any matrix $W$ and vector $z$, and the last inequality follows since $\rho(B_i) = \lim \|B_i^k\|^{\frac{1}{k}}$ and hence for any $\mu \in (\rho(B_i), 1)$ we have that $\|B_i^k\| \leq \mu^{k}$ for all large enough $k$.

Now we turn our attention to the lower bound. Given a vector $z$, we define $z_i$ to be the vector with the two components that multiply the block $B_i$ in the matrix-vector product $Qz$. Using the same expansion as above, we get
\begin{align*}
\|z^{K+k+1} -z^{K+k} - v\|_2^2 &= \|Q^k(((Q-I)z^K-p)-Q^\infty ((Q-I)z^K-p)\|_2^2\\
&= \sum_{i \in [n] } \|(B_i^k - B_i^\infty)((Q-I)z^K-p)_i\|_2^2\\
&= \sum_{i \in \Delta } \|B_i^k((Q-I)z^K-p)_i\|_2^2\\
&\geq \sum_{i \in \Delta } \sigma_{\min}(B_i^k)^2\| ((Q-I)z^K-p)_i\|_2^2\\
&\geq \min_{i \in \Delta } \sigma_{\min}(B_i^k)^2 \sum_{i \in \Delta }\| ((Q-I)z^K-p)_i\|_2^2\\
&= \min_{i \in \Delta } \sigma_{\min}(B_i^k)^2 \| (I-Q^\infty)((Q-I)z^K-p)\|_2^2\\
&= \left(\min_{i \in \Delta } \sigma_{\min}(B_i)\right)^{2k} \| (Q-I)z^K+(Q^\infty-I)p\|_2^2
\end{align*}
where the second equality and the penultimate one use the block-diagonal structure of the matrix $Q$ to decompose the norm of the matrix-vector product into orthogonal components, and the first inequality follows since $B_i^k$ is a rank two matrix for any $i \in \Delta$.

\textbf{Normalized iterates}. We now turn our attention to the normalized iterates. The upper bound follows almost immediately from Theorem~\eqref{thm:general-convergence} if we consider the PDHG algorithm started at $z^K$. To show the bound with the theorem, it suffices to note that 1) in this case $v = -Q^\infty p$ and so we might pick $z^\star = (Q-I)^\dagger (I-Q^\infty)p$, and 2) all the norms in finite dimensional spaces are equivalent so we can upper bound $\|\cdot\|_M \leq C\|\cdot\|_2$ for some constant $C > 0$. To see the first point note that
\begin{align*}
        Qz^\star - p - z^\star = -Q^\infty p & \iff (Q-I) z^\star = (I - Q^\infty)p \\
        & \Longleftarrow z^\star = (Q-I)^\dagger (I-Q^\infty)p\ ,
\end{align*}
where $(I-Q)^\dagger$ is the pseudo inverse of $(I-Q).$

Now we turn our attention to the lower bound. Just as before we analyze the dynamics of $Q^k$ by studying the individual blocks $B_i^k$. We will use the following identity for blocks satisfying $\rho(B) < 1$:
\begin{equation}\label{eq:neumann}
    \sum_{j=0}^k B^j = (I-B)^{-1} (I-B^{k+1})\ .
\end{equation}
Recall that $\Delta = \{i \mid \sigma_i > 0 \}$, which corresponds with blocks satisfying $\rho(B_i) < 1$. Additionally, recall that $p_i$ is the vector with the two components of $p$ that multiply the block $B_i$ in the matrix-vector product $Qp$.
Expanding we get
\begin{align*}
    \left\|v - \frac{1}{k}z^{K+k+1}\right\|^2 & = \left\|Q^\infty  p + \frac{1}{k}\sum_{j=0}^{k}Q^j p - Q^{k+1}z^K\right\|^2\\
    &=\sum_{i\in [n]} \left\|B_i^\infty p_i + \frac{1}{k}\sum_{j=0}^k B_i^j p_i - \frac{1}{k} B_i^{k+1} z^K_i\right\|^2\\
    &= \frac{1}{k^2}\sum_{i\in \Delta} \left\| (I-B_i)^{-1} (I-B_i^{k+1}) p_i - B_i^{k+1}z^K_i\right\|^2 + \frac{1}{k^2}\sum_{i\notin \Delta}\| B_i^{k+1} z^K_i\|^2\\
    & \geq \frac{1}{k^2}\sum_{i\in \Delta}(\sigma_{\max}(I-B_i))^{-2}\left\| (I-B_i^{k+1}) p_i - (I-B_i)B_i^{k+1} z^K_i\right\|^2
\end{align*}
where for the last two equalities we used the fact that $Q$ is block diagonal, and the last inequality follows since $I-B_i$ is invertible for $i \in \Delta$. Then, taking the minimum coefficient we get
\begin{align*}
  &\frac{1}{k^2}\sum_{i\in \Delta}\sigma_{\max}(I-B_i)^{-2}\left\| (I-B_i^{k+1}) p_i - (I-B_i)B_i^{k+1} z^K\right\|^2 \\
  &  \hspace{4cm} \geq\frac{1}{k^2}\min_{j \in \Delta}\left\{\sigma_{\max}(I-B_j)^{-2}\right\}\sum_{i\in \Delta} \left\| (I-B_i^{k+1})  p_i - (I-B_i)B_i^{k+1} z^K\right\|^2\\
   &\hspace{4cm}  =\frac{1}{k^2}\min_{j \in \Delta}\left\{\sigma_{\max}(I-B_j)^{-2}\right\} \left\| (I-Q^{k+1}) p - (I-Q) Q^{k+1}z^K\right\|^2\\
   &\hspace{4cm}  \geq \frac{1}{2k^2} \min_{j \in \Delta}\left\{\sigma_{\max}(I-B_j)^{-2}\right\} \left\| (I-Q^{\infty}) p\right\|^2\ ,
\end{align*}
where the last equality uses the fact that the matrices we handle are block diagonal, and the last line follows since $\|(I-Q^{k+1})p - (I-Q)Q^{k+1}z^K\| \rightarrow \|(I-Q^\infty)p\|$, and so the inequality holds for sufficiently large $k \geq 0.$

\textbf{Normalized average}. The upper bound follows from the exact same argument as the normalized iterates by using Theorem~\ref{thm:general-convergence}.

The lower bound for this case is sightly more intricate. We expand and apply \eqref{eq:neumann}
\begin{align*}
    &\left\|v - \frac{2}{k(k+1)}\sum_{j=1}^k z^{K+j} \right\|^2
    \\
    &\hspace{1cm}= \left\|Q^\infty p - \frac{2}{k(k+1)}\sum_{j=1}^{k}\sum_{l=0}^{j-1}Q^l p + \frac{2}{k(k+1)}\sum_{j=1}^k Q^jz^{K}\right\|^2 \\
     & \hspace{1cm}= \sum_{i \in [n]}\left\|B_i^\infty  p_i -  \frac{2}{k(k+1)}\sum_{j=1}^{k}\sum_{l=0}^{j-1}B_i^l  p_i +  \frac{2}{k(k+1)}\sum_{j=1}^k B^jz^{K}_i\right\|^2\\
    &\hspace{1cm} = \sum_{i \in \Delta}\left\| \frac{2}{k(k+1)}\sum_{j=1}^{k}\sum_{l=0}^{j-1}B_i^l  p_i +  \frac{2}{k(k+1)}\sum_{j=1}^k B^j_iz^{K}_i\right\|^2 +  \sum_{i\notin \Delta}\left\|\frac{2}{k(k+1)}\sum_{j=1}^k B^j_i z^{K}_i\right\|^2\\
    &\hspace{1cm} = \frac{4}{k^2(k+1)^2}\sum_{i \in \Delta}\left\| (I-B_i)^{-1}\left(\sum_{j=1}^{k}(I - B_i^{j})  p_i + (I - B^{k+1}_i) z^{K}_i\right) - z_i^K\right\|^2+  \sum_{i\notin \Delta}\left\|\frac{2}{k(k+1)}\sum_{j=1}^k B^j_i z^{K}_i\right\|^2.
\end{align*}
The second equality follows from the fact that $Q$ and $Q^\infty$ are block diagonal.
Then, dropping the second sum and using the fact that $\|(I-B_i)^{-1}z\|_2 \geq \sigma_{\max}(I-B_i)\|z\|_2$ for all $z$ and $i \in \Delta$, we can lower bound
\begin{align*}
&\frac{4}{k^2(k+1)^2}\sum_{i \in \Delta}\left\| (I-B_i)^{-1}\left(\sum_{j=1}^{k}(I - B_i^{j})  p_i + B_i (I - B_i^k) z^{K}_i\right)\right\|^2+  \sum_{i\notin \Delta}\left\|\frac{2}{k(k+1)}\sum_{j=1}^k B^j_i z^{K}_i\right\|^2 \\
    &\hspace{1cm} \geq \frac{4}{k^2(k+1)^2}\sum_{i \in \Delta}\sigma_{\max}(I-B_i)^{-2}\left\|\sum_{j=1}^{k}(I - B_i^{j})  p_i + B_i (I - B_i^k) z^{K}_i\right\|^2\\
    &\hspace{1cm} \geq \frac{4}{k^2(k+1)^2}\min_{i \in \Delta}\left\{\sigma_{\max}(I-B_i)^{-2}\right\}\sum_{i \in \Delta}\left\|k p_i - (I-B_i)^{-1} B_i (I-B_i^k)  p_i + B_i (I - B_i^k) z^{K}_i\right\|^2\\
    &\hspace{1cm} \geq \frac{4}{k^2(k+1)^2}\min_{i \in \Delta}\left\{\sigma_{\max}(I-B_i)^{-4}\right\}\sum_{i \in \Delta}\left\|k (I-B_i)  p_i - B_i (I-B_i^k)  p_i + (I-B_i) B_i (I - B_i^k) z^{K}_i\right\|^2\\
    &\hspace{1cm} = \frac{4}{k^2(k+1)^2}\min_{i \in \Delta}\left\{\sigma_{\max}(I-B_i)^{-4}\right\}\left\|k (I-Q) p - Q (I-Q^k) p + (I-Q) Q (I - Q^k) z^{K}\right\|^2\\
    &\hspace{1cm} = \frac{4}{(k+1)^2}\min_{i \in \Delta}\left\{\sigma_{\max}(I-B_i)^{-4}\right\}\left\| (I-Q) p - \frac{1}{k}Q(I-Q^k) p + \frac{1}{k}(I-Q)Q(I - Q^k) z^{K}\right\|^2\\
    &\hspace{1cm} \geq \frac{2}{(k+1)^2}\min_{i \in \Delta}\left\{\sigma_{\max}(I-B_i)^{-4}\right\}\left\|(I-Q) p\right\|^2,
\end{align*}
where the last inequality follows for large enough $k$ since $\left\| - \frac{1}{k}Q(I-Q^k) p + \frac{1}{k}(I-Q)Q(I - Q^k) z^{K}\right\|^2 \rightarrow 0$. This completes the proof.
\end{proof}
\section{Numerical experiments}\label{sec:experiments}

In this section, we test numerically the proposed approach to check infeasibility using PDHG. For the experiments, we implemented PDHG for LP problems of the form
\begin{equation}\label{exp-primal-dual}
    \begin{aligned}[c]
      \text{minimize} &\quad c^\top x\\
    \text{subject to} &\quad Ax  \geq b\\
    &\quad l \leq x \leq u
    \end{aligned}
    \qquad\qquad\qquad
    \begin{aligned}[c]
    \text{maximize}  &\quad b^\top y + l^\top r_+ - u^\top r_-\\
    \text{subject to} &\quad c -  A^\top y  = r\\
    &\quad y \geq 0
    \end{aligned}
\end{equation}
where $b \in \RR^m, l\in (\RR \cup \{-\infty\} )^n, u \in (\RR \cup \{\infty\})^n, A \in \RR^{m\times n}$ are given and $r_+~=~\proj_{\RR_+}(r)$ and $r_-~=~-\proj_{\RR_+}(-r)$ are the projections of $r$ onto the positive and negative orthant, respectively. We chose this form over the standard form \eqref{primal}-\eqref{dual} since it is algorithmically easier to reduce arbitrary LP problems to it. Given that the PDHG algorithm for this formulation generates iterates $(x^k, y^k)$, for our computations we generate $r^k$ by finding the closest point to $c-A^\top y^k$ that makes the dual objective finite; i.e., $r^k = \proj_{C_v}(c-A^\top y^k)$ with $C_v$ defined below. All the results proved in this work also apply to this form under suitable modifications of the statements.

For our experiments we use the \emph{Netlib dataset of infeasible LP instances} \cite{netlibinf}. We use this dataset to illustrate the different dynamics that PDHG exhibits. For all our experiments we measure statistics that quantify how close are the candidate iterates \eqref{eq:sequences} to being approximate certificates of inteasibility.

\begin{figure}[t]
    \centering
   \begin{subfigure}[b]{0.45\textwidth}
        \includegraphics[width=\textwidth]{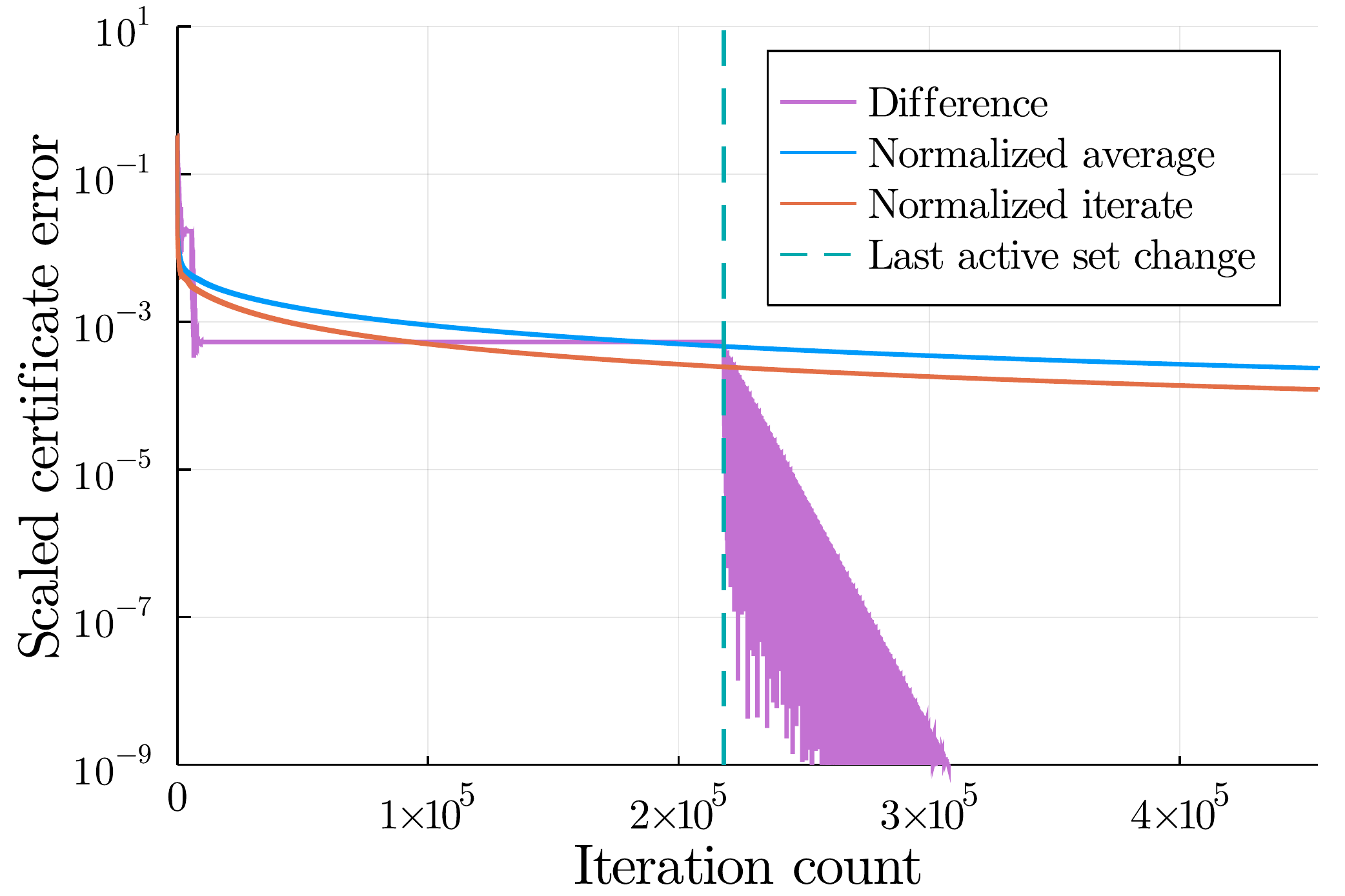}
        \caption{\texttt{box1}}
    \end{subfigure}
    \begin{subfigure}[b]{0.45\textwidth}
      \includegraphics[width=\textwidth]{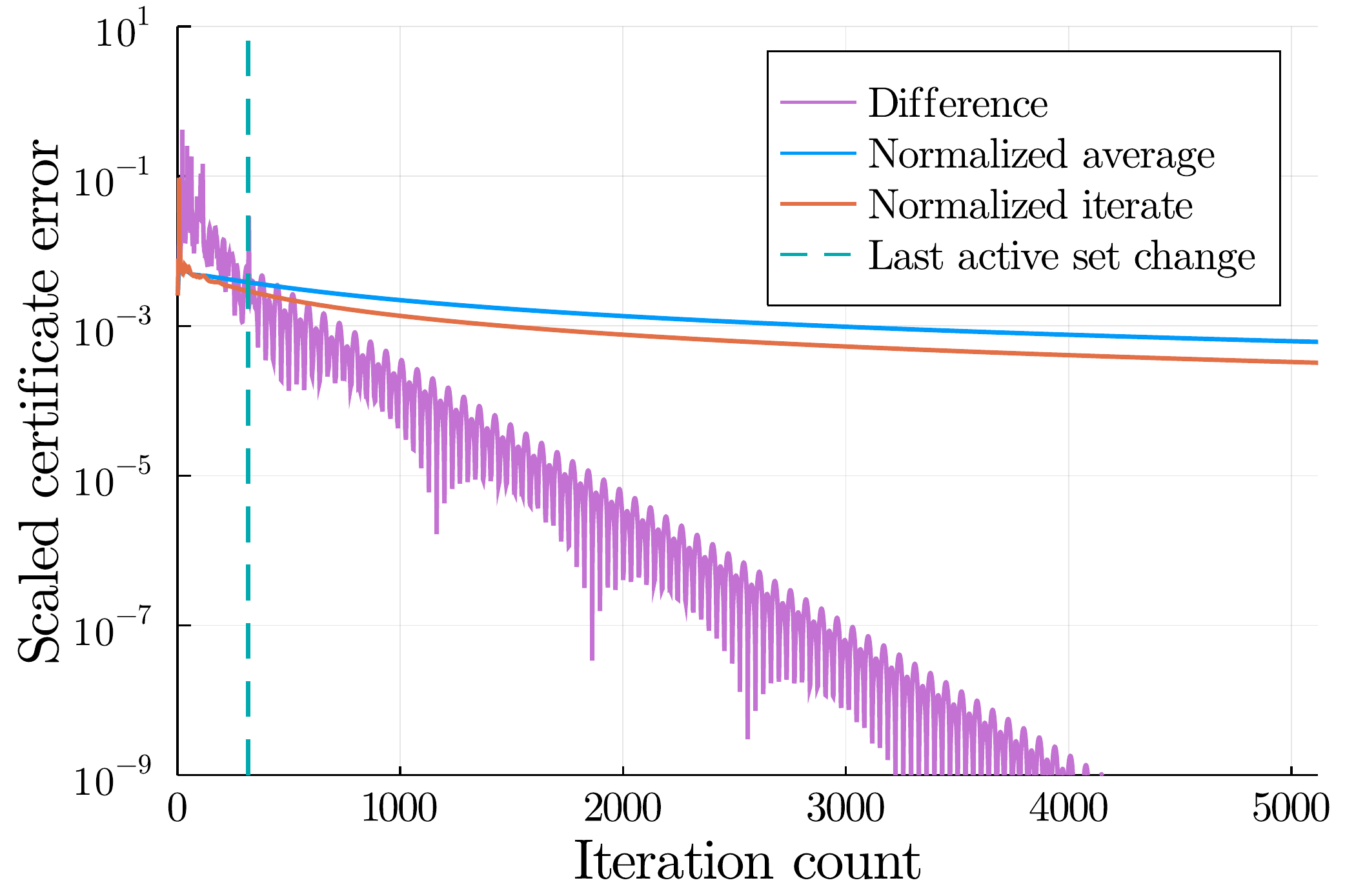}
    \caption{\texttt{woodinfe}}
    \end{subfigure}
      \begin{subfigure}[b]{0.45\textwidth}
        \includegraphics[width=\textwidth]{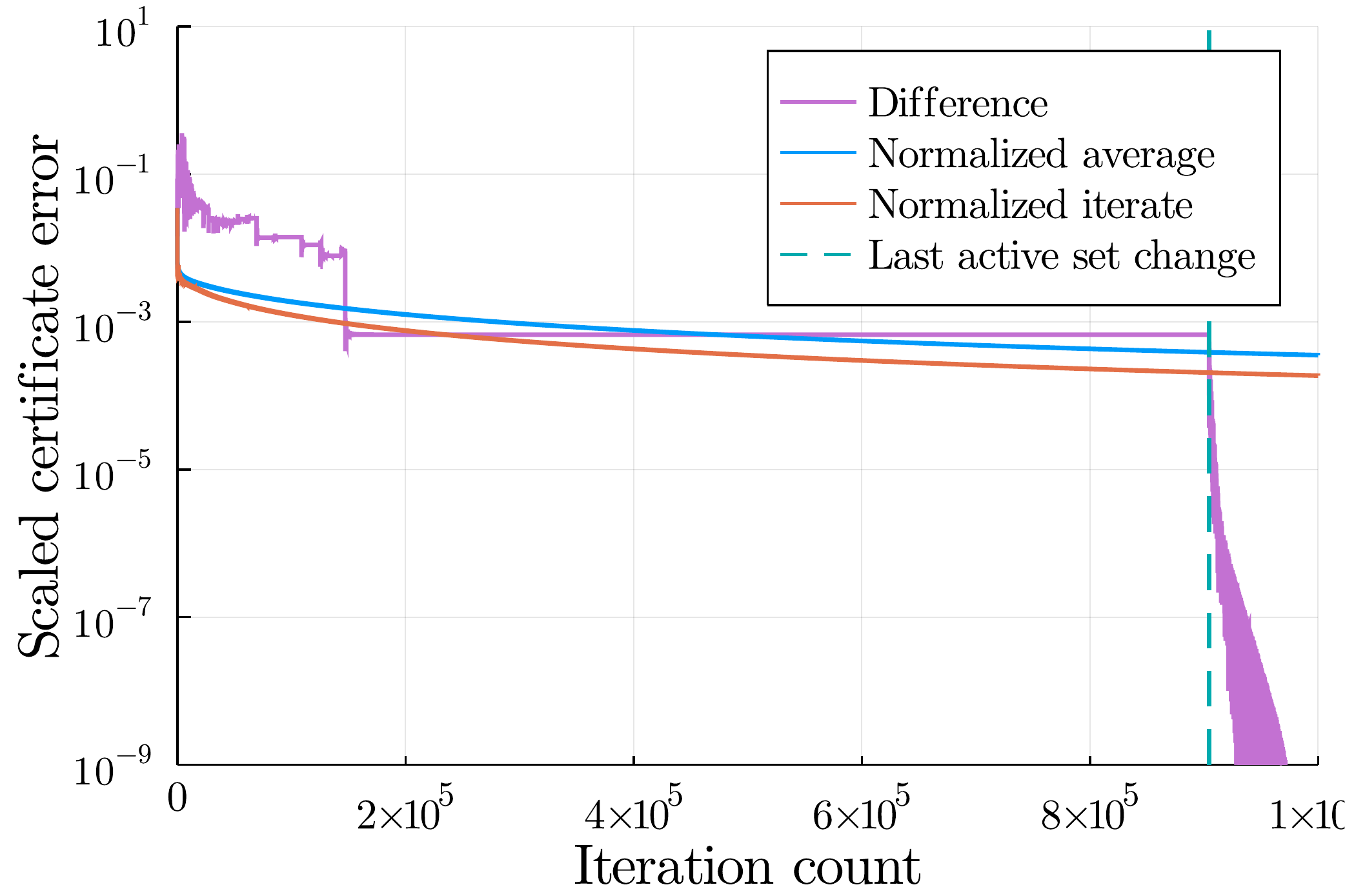}
        \caption{\texttt{ex72a}}
            \end{subfigure}
    \begin{subfigure}[b]{0.45\textwidth}
        \includegraphics[width=\textwidth]{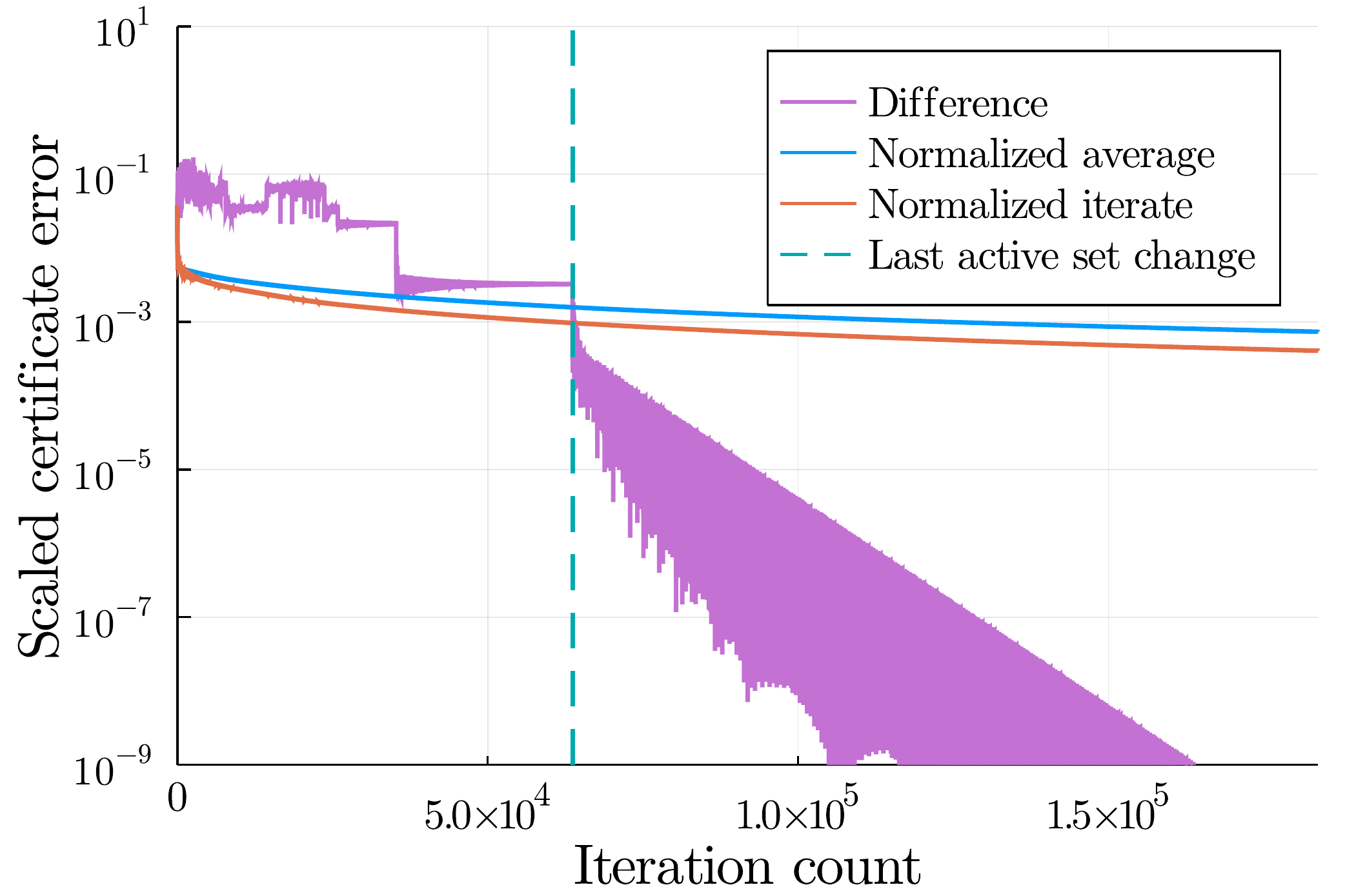}
        \caption{\texttt{ex73a}}
            \end{subfigure}
            \caption{Scaled certificate error \eqref{eq:feasible_error} for the three sequences defined in \eqref{eq:sequences} for four instances of the Netlib infeasible dataset \cite{netlibinf}. Vertical dotted lines denote the last observed active set change.}
\label{fig:identifiability}
\end{figure}
Before we describe these statistics, let us define what we mean by approximate infeasibility certificates. The set of (exact) primal infeasibility certificates for \eqref{exp-primal-dual} is given by all the vectors $(y, r)\in \RR_+^m \times \RR^n$ satisfying
\begin{equation}
 b^\top y + l^\top r_+ -  u^\top r_- > 0 \text{ and } r = - A^\top y
\end{equation}
while the set of (exact) dual infeasibility certificates is given by all the vectors $x\in\RR^n$ satisfying
\begin{equation}
    c^\top x < 0,\quad x \in C_v,  \quad  \text{ and }\quad Ax \geq 0
\end{equation}
where the set $C_v$ is given by
\begin{equation*}\small
    C_v = \left\{x : \text{ for $i\in [n]$, }  x_i \begin{cases} = 0 & \text{if } l_i, u_i \in \RR, \\ \geq 0 &\text{if } l_i \in \RR, u_i = \infty, \\ \leq 0 &\text{if }l_i =-\infty, u_i \in \RR,\end{cases} \right \}\ .
\end{equation*}
We define an $\varepsilon$-approximate primal infeasibility certificate to be any point $(y, r) \in \RR_+^m \times \RR^n$ satisfying
\begin{equation}
    b^\top y +  l^\top r_+ -   u^\top r_- > 0 \quad \text{and}\quad (b^\top y + l^\top r_+ -  u^\top r_- )^{-1} \|r+A^\top y\|_\infty \leq \varepsilon.
\end{equation}
Similarly we say that a point $x\in \RR^m$ is an $\epsilon$-approximate dual infeasibility certificate if it satisfies
\begin{equation}
    c^\top x < 0, \quad \frac{1}{-c^\top x} \cdot \|x - \proj_{C_v}(x)\|_\infty \leq \varepsilon, \qquad \text{and} \qquad \frac{1}{-c^\top x} \cdot \|Ax - \proj_{\RR_+^m}(Ax)\|_\infty \leq \varepsilon.
\end{equation}
These definitions parallel the criteria to detect infeasibility used by SCS \cite{SCS}, a popular open-source solver.

Since all the instances in the Netlib infeasible data set are primal infeasible, we will only plot information about the dual components of the candidate certificates \eqref{eq:sequences}. To illustrate how close is each candidate to being a certificate we will plot
\begin{equation}\label{eq:feasible_error}
    \frac{\|r^k+A^\top y^k\|_\infty}{b^\top y^k + l^\top r^k_+ -u^\top r^k_-} .
\end{equation}
We call this quantity the \emph{scaled certificate error}. If the objective term, i.e., the denominator, is negative at any iteration then we do not plot it for that iteration. For almost all the problems we consider \eqref{eq:feasible_error} remains positive for almost all iterations.

\textbf{Nondegeneracy in practice}. Our first batch of experiments showcases the faster convergence of the difference of iterates in practice. We found empirically that for a subset of instances in the dataset, the difference of iterates \eqref{eq:diff_iterates} detects infeasibility faster than the other two sequences. Based on the theory, we expect that for these instances, the difference exhibits eventual faster convergence. To test this claim, we run an experiment on four of these instances: \texttt{box1}, \texttt{woodinfe}, \texttt{ex73a} and \texttt{ext72a}.

Figure \ref{fig:identifiability} displays the scaled certificate error \eqref{eq:feasible_error} against the number of iterations for the four instances. For all of them, we can see a clear phase transition between a first stage of slow convergence and a second stage that displays linear convergence. This transition is unequivocally marked by the last change of the active set of the solution (also depicted in the figure). Notice, however, that the iteration number at which the active set is fixed might be large; the point at which this happens ranges among multiple orders of magnitude in our experiments.

\textbf{Normalized iterates can be faster}.
Even if eventual identifiability holds, this might take a significant number of iterations. In these cases it might be beneficial to check infeasibility using the normalized iterated \eqref{eq:normalized_iterates} and the normalized average \eqref{eq:normalized_average}. In this batch of experiments we run PDHG on \texttt{bgdbg1} and \texttt{chemcom}, the results are displayed in Figure \ref{fig:iterates-better}. Just as before we plot the scaled certificate error against the number of iterations.

The normalized average is consistently slower at converging than the normalized iterates. This is most likely due to the fact that it retains a tail of initial iterates, which are far away from the infimal displacement vector. For both these problems, the difference takes at least twice the number of iterations than the normalized iterates to obtain a highly accurate certificate, i.e., $\varepsilon = 10^{-8}$. This suggests that solvers may benefit from checking infeasibility with both the normalized iterates \eqref{eq:normalized_iterates} and difference of iterates \eqref{eq:diff_iterates}.

\begin{figure}[t]
   \centering
\begin{subfigure}[b]{0.45\textwidth}
\includegraphics[width=\textwidth]{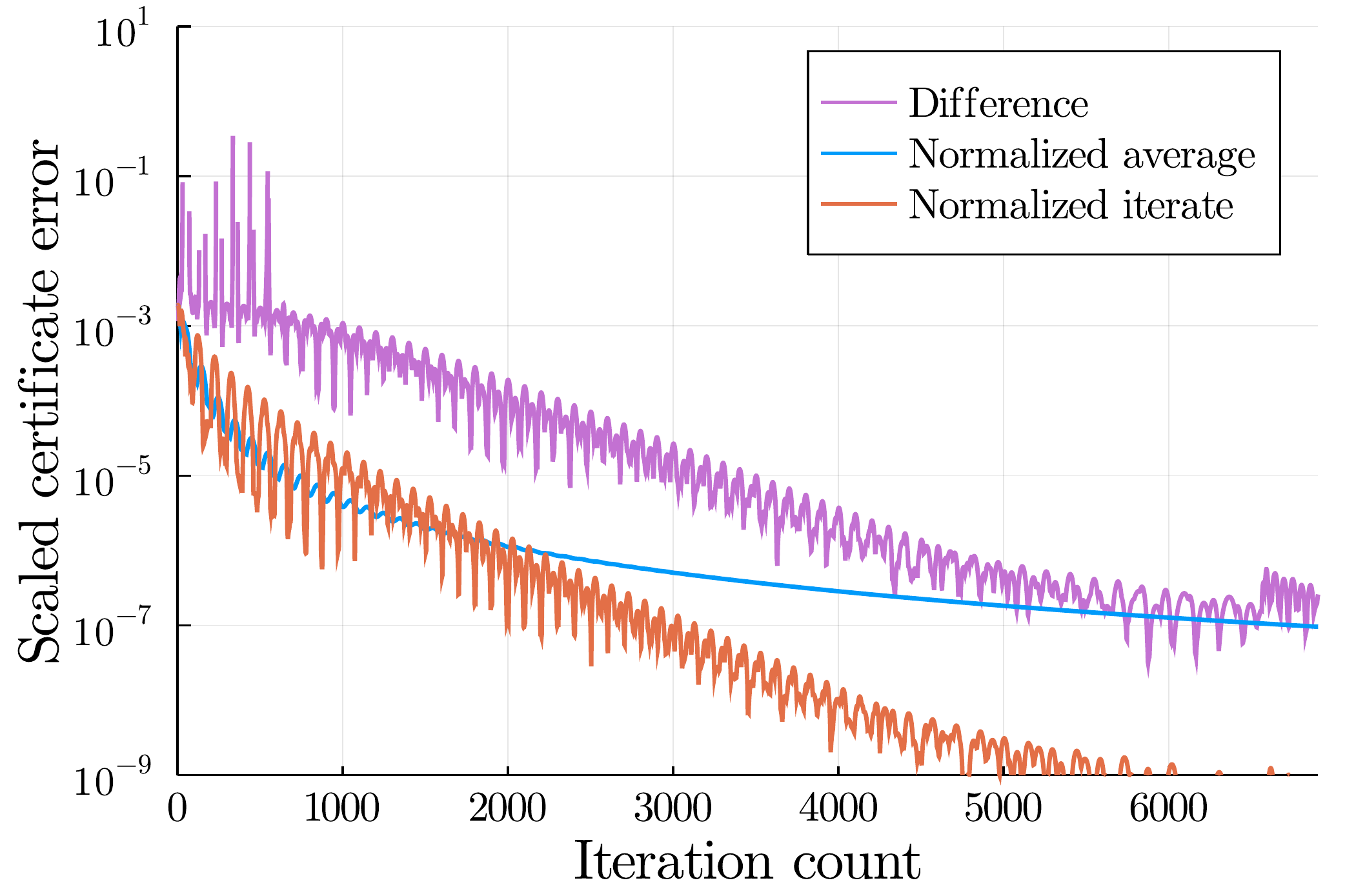}
\caption{\texttt{bgdbg1}}
\end{subfigure}
\begin{subfigure}[b]{0.45\textwidth}
\includegraphics[width=\textwidth]{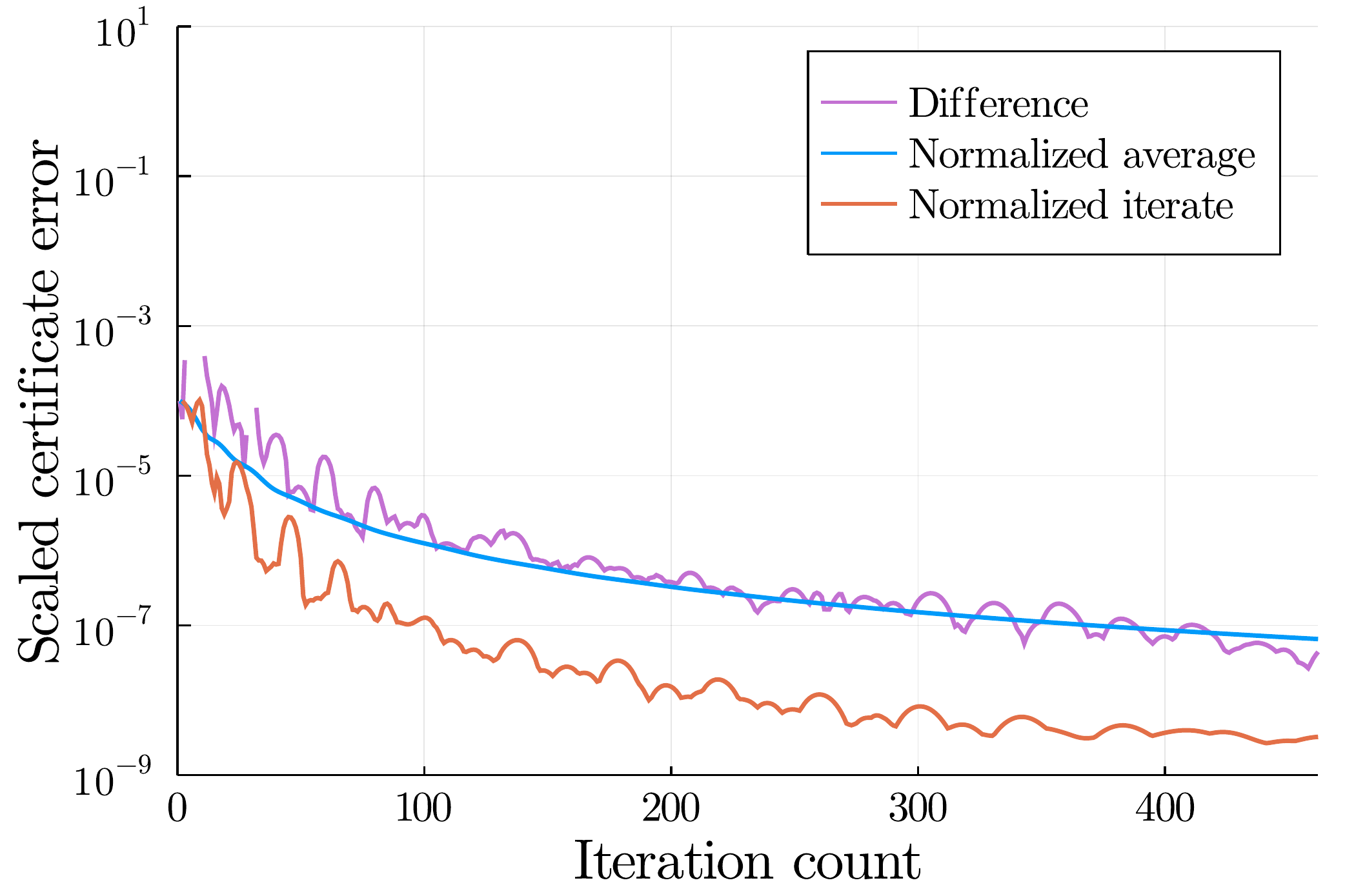}
\caption{\texttt{chemcom}}
\end{subfigure}
\caption{Scaled certificate error \eqref{eq:feasible_error} for the three sequences defined in \eqref{eq:sequences} for four instances of the Netlib infeasible dataset \cite{netlibinf}.}\label{fig:iterates-better}
\end{figure}

\section{Conclusions and future directions}
\label{sec:conclusions}
In this work, we showed how to detect infeasibility of LP programs using the iterates of PDHG. The proposed approach is simple, doesn't require changing the logic of the algorithm, and has an almost negligible computational cost. Thus, it is suitable for large-scale LP problems. The ideas developed in this work pave the road for a number of research directions. Below we list some of these directions and superficially address their main difficulties.

It would be interesting to understand how to integrate our theoretical results with enhanced versions of PDHG. For example, for practical applications it is often useful to have varying stepsizes $\pssize_k$ and $\dssize_k$ that compensate for the different scales that the primal and dual problems might have \cite{malitsky2018first}. In this case, there is no unique PDHG operator $T$ applied at each iteration, rather a sequence of operators $T_k$ and even the definition of $v$ is not well-posed.

In a similar vein, we wonder if it is possible to speed up the convergence of some of the sequences in \eqref{eq:sequences} using a restart scheme. This is known to be the true in the feasible case. Tackling this question would require designing a more involved method that restarts the PDHG iterates and the candidate certificates periodically. It is unclear if it is possible to have a scheme that works seamlessly for both the feasible and infeasible cases.

The proofs of most of our results rely on the form of the PDHG updates. For instance, one of our key arguments established a connection between the iterates of PDHG running on an infeasible problem and the iterates of PDHG running on an auxiliary feasible problem. It is natural to wonder if this is a phenomenon that only holds for PDHG or if it  generalizes to other algorithms and/or other problem families such as quadractic and semidefinite programming.

Another potential research direction is the application of PDHG for solving the homogeneuous self-dual embedding problem akin to \cite{o2016conic}. A naive implementation of this idea would require doubling the dimension of the primal problem. The authors of \cite{o2016conic} showed that for ADMM, this increase in the dimension doesn't affect the computational complexity of each iteration. At its core, their argument for this fact used the symmetry between primal and dual updates in ADMM; it is unclear how to extend this argument to PDHG given its asymmetric updates.

\section*{Acknowledgements}

We would like to thank Brendan O'Donoghue, Oliver Hinder, and Warren Schudy for insightful conversations during the first stages of this work.
\bibliographystyle{plain}
\bibliography{bibliography2}

\appendix
\section{Additional proofs}
\subsection{Proof of Proposition~\ref{prop:convergence-implications}}\label{sec:proof-prop-1}
\begin{proof}
Assume that $(z^{k+1} - z^k) \rightarrow v $. Fix $\varepsilon > 0$. Our goal is to show that for all $k$ large enough $\|z^k/k - v\| \leq \varepsilon$. Due to convergence, there exist $K_1 \in \NN $ such that for all $k \geq K_2$ we have $\|z^{k+1} - z^k - v\| \leq \varepsilon/3$. Define
$$B := \max_{k \leq K_1} \|z^{k+1} - z^k - v\|, $$
let $K_2\in \NN$ be such that for all $k \geq K_2$ we get $K_1 B/k \leq \varepsilon/3$, and let $K_3 \in \NN$ be such that if $k \geq K_3$ then $\|z^0\|/k \leq \varepsilon /3$.  Then, for any $k \geq  \max\{K_1, K_2, K_3\}$ we have
 \begin{align*}
     \left\|\frac{z^k}{k}-v \right\| &\leq \left\|\frac{1}{k}(z^k - z^0) -v \right\|  + \frac{1}{k}\|z^0\|\\
     & = \left\|\frac{1}{k}\sum_{j=1}^k (z^j - z^{j-1}) -v \right\|  + \frac{1}{k}\|z^0\|\\
    & \leq \frac{1}{k}\sum_{j=1}^k\left\| (z^j - z^{j-1}) -v \right\|  + \frac{1}{k}\|z^0\|\\
    & \leq \frac{K_1}{k} B + \frac{1}{k}\sum_{j=K_1}^k\left\| (z^j - z^{j-1}) -v \right\|  + \frac{1}{k}\|z^0\|\\
    & \leq  \frac{2}{3} \varepsilon + \frac{1}{3k}\sum_{j=K_1}^k \varepsilon \leq \varepsilon\ .
 \end{align*}
 This proves the first statement.

 Now, assume that $\frac{z^k}{k} \rightarrow v$, and fix $\varepsilon > 0$. Just as before define $K_1 \in \NN$ to be such that for all $\|z^k/k - v\| \leq  \varepsilon/2$, define the constant $B = \max_{k \leq K_1}\|z^k/k - v\|$, and let $K_2$ be such that $B(K_1 +1)K_1/((K_2+1)K_2) \leq \varepsilon / 2$. Then, we have that for any $k \geq \max\{K_1, K_2\}$,
 \begin{align*}
     \left\|\frac{2}{(k+1)k}\sum_{j = 1}^k {z^j}-v \right\| &= \frac{2}{(k+1)k}\left\|\sum_{j = 1}^k {z^j}-\frac{k(k+1)}{2}v\right\|\\
    &= \frac{2}{(k+1)k}\left\|\sum_{j = 1}^k ({z^j}-jv)\right\|\\
    & \leq \frac{2}{(k+1)k}\sum_{j = 1}^k  j \left\|\frac{z^j}{j}-v\right\|\\
    & \leq \frac{(K_1+1)K_1}{(k+1)k}B +  \frac{2}{(k+1)k}\sum_{j = K_1}  j \left\|\frac{z^j}{j}-v\right\| \\
    & \leq  \frac{\varepsilon}{2} + \frac{\varepsilon}{2} \left(\frac{2}{(k+1)k}\sum_{j=K_1}^k  j\right) \leq \varepsilon\ .
 \end{align*}
\end{proof}
\section{Counterexamples} \label{sec:counterexamples}
\begin{example}[\textbf{Differences don't converge, but normalized iterates do}] \label{counterexample-diff-iterates} Consider the sequence $(z^k) \subseteq \RR^2$ generated by applying a $90^\circ$ counterclockwise rotation repeatedly starting from $z^0 = e_1 = (1, 0)$. Thus, the sequence cycles as  $$z^1=(0, 1), z^2 = (-1, 0), z^3 = (0, -1), z^4 = (1, 0), \dots\ .$$
For this example, the differences of iterates $z^{k+1}-z^k$ also cycle among four possibilities and do not converge. Nonetheless, since the iterates are bounded $\frac{1}{k}z^k \rightarrow 0.$
\end{example}
\begin{example}[\textbf{Normalized iterates diverge, but normalized averages converge}] Consider the sequence $(z^k) \subseteq \RR^2$ given by $z^k = (-1)^k k^\frac{3}{2}$ with $k \in \NN.$ Then, it is clear that $|z^k|/k > \sqrt{k}$, and so the normalized iterates diverge. On the other hand, notice that $$\frac{2}{(k+1)} \bar{z}^k = \sum_{j=1}^k (-1)^j \frac{2k^{\frac{1}{2}}}{k+1},$$
and it is easy to show that this series converges using the Leibniz Test.
\end{example}
\begin{example}[\textbf{Nonexpansive operator with divergent $z_\epsilon$}]\label{ex:divergent} Let $T: \RR \rightarrow \RR$ given by
\begin{equation}
    T(z) = z + f(z) \qquad \text{where} \qquad f(z) = \begin{cases}\exp(-z^2) + 1 & \text{if }z > 0, \text{ or} \\ 2 & \text{otherwise.}\end{cases}
\end{equation}
Since the derivative of $T$ is bounded by $1$, we get that $T$ is a nonexpansive operator. Furthermore,  $\range(T-I) = \range(f) = (1, 2]$, and so $v = 1$. If we define $z_\epsilon$ to be a point such that $|v -(T-I)(z_\epsilon)| \leq \epsilon$, we see that $z_\epsilon > \Omega\left(\log\left(\frac{1}{\epsilon}\right)^{\frac{1}{2}}\right)$, and thus it diverges as $\epsilon \rightarrow 0$.
\end{example}
\end{document}